\documentclass[twoside,11pt]{amsart} 
\usepackage{epsfig}
\usepackage{graphicx}
\usepackage{latexsym} 
\usepackage{amsmath} \usepackage{amssymb}
\usepackage{cancel}

 \keywords{ primes, twin primes, gaps, prime constellations, Eratosthenes sieve,
primorial numbers, Polignac's conjecture}

\subjclass{11N05, 11A41, 11A07}

\newtheorem{theorem}{Theorem}[section]
\newtheorem{lemma}[theorem]{Lemma}
\newtheorem{corollary}[theorem]{Corollary}
\newtheorem{remark}[theorem]{Remark}

\newdimen\epsfxsize
\newdimen\epsfysize

\newcommand {\gap}     {\makebox[0.075 in]{}}   
\newcommand {\biggap}     {\makebox[0.15 in]{}}   
\newcommand {\st}      {\gap : \gap}

\newcommand {\fto}     {\longrightarrow}

\newcommand {\set}[1]  {\left\{ {#1} \right\}}   
\newcommand {\ord}[1]  {{#1}^{\rm th}}

\newcommand{\pml}[1] {{#1}^\#}

\newcommand{\Z}     {{\mathbb Z}}
\newcommand{\Zmod}  {\Z \bmod \pml{p_k}}
\newcommand{\Zmodp}  {\Z \bmod \pml{p_{k+1}}}

\newcommand{\N}[2]  {N_{{#2}}({#1})}
\newcommand{\Rat}[2]  {w_{{#2},{#1}}}   

\newcommand{\pgap}   {{\mathcal G}}

\newcommand{\MJP}[2] {\left. M_{{#1}} \right|_{{#2}}}
\newcommand{\AJP}[2] {\left. \Lambda_{{#1}} \right|_{{#2}}}
\newcommand{\wgp}[2] {\left. \bar{w}_{{#1}} \right|_{{#2}}}

\newcommand{\Bi}[2]{\left( \begin{array}{c}{#1} \\ {#2} \end{array}\right)}
\newcommand{\lil}{\scriptstyle}

\begin{document}

\title{Eratosthenes~sieve and~the~gaps~between~primes}

\date{25 Aug 2014}

\author{Fred B. Holt and Helgi Rudd}
\address{fbholt@uw.edu ;  4311-11th Ave NE \#500, Seattle, WA 98105;
48B York Place, Prahran, Australia 3181}

\begin{abstract}
A few years ago we identified a recursion that works directly with the gaps among
the generators in each stage of Eratosthenes sieve.  This recursion provides explicit
enumerations of sequences of gaps among the generators, which sequences are known as 
constellations.

By studying this recursion on the cycles of gaps 
across stages of Eratosthenes sieve, we are able to provide evidence on a number of
open problems regarding gaps between prime numbers.  The basic counts of short 
constellations in the cycles of gaps provide evidence toward the twin prime conjecture
and toward resolving a series of 
questions posed by Erd{\" o}s and Tur{\'a}n.  The dynamic system underlying
the recursion provides evidence toward Polignac's conjecture and in support of the
estimates made for gaps among primes by Hardy and Littlewood
in Conjecture B of their 1923 paper.

\end{abstract}

\maketitle

\section{Introduction}
 
We work with the prime numbers in ascending order, denoting the
$\ord{k}$ prime by $p_k$.  Accompanying the sequence of primes
is the sequence of gaps between consecutive primes.
We denote the gap between $p_k$ and $p_{k+1}$ by
$g_k=p_{k+1}-p_k.$
These sequences begin
$$
\begin{array}{rrrrrrc}
p_1=2, & p_2=3, & p_3=5, & p_4=7, & p_5=11, & p_6=13, & \ldots\\
g_1=1, & g_2=2, & g_3=2, & g_4=4, & g_5=2, & g_6=4, & \ldots
\end{array}
$$

A number $d$ is the {\em difference} between prime numbers if there are
two prime numbers, $p$ and $q$, such that $q-p=d$.  There are already
many interesting results and open questions about differences between
prime numbers; a seminal and inspirational work about differences
between primes is Hardy and Littlewood's 1923 paper \cite{HL}.

A number $g$ is a {\em gap} between prime numbers if it is the difference
between consecutive primes; that is, $p=p_i$ and $q=p_{i+1}$ and
$q-p=g$.
Differences of length $2$ or $4$ are also gaps; so open questions
like the Twin Prime Conjecture, that there are an infinite number
of gaps $g_k=2$, can be formulated as questions about differences
as well.

A {\em constellation among primes} \cite{Riesel} is a sequence of consecutive gaps
between prime numbers.  Let $s=a_1 a_2 \cdots a_k$ be a sequence of $k$
numbers.  Then $s$ is a constellation among primes if there exists a sequence of
$k+1$ consecutive prime numbers $p_i p_{i+1} \cdots p_{i+k}$ such
that for each $j=1,\ldots,k$, we have the gap $p_{i+j}-p_{i+j-1}=a_j$.  
Equivalently,
$s$ is a constellation if for some $i$ and all $j=1,\ldots,k$,
$a_j=g_{i+j}$.

We do not study the gaps between primes directly.  Instead, we study the cycle of gaps
$\pgap(\pml{p})$ at each stage of Eratosthenes sieve.  Here, $\pml{p}$ is the
{\em primorial} of $p$, which is the product of all primes from $2$ up to and including $p$.
$\pgap(\pml{p})$ is the cycle of gaps among the generators of $\Z \bmod \pml{p}$.
These generators and their images through the counting numbers are the candidate primes
after Eratosthenes sieve has run through the stages from $2$ to $p$.  All of the remaining primes
are among these candidates.

There is a substantial amount of structure preserved in the cycle of gaps from
one stage of Eratosthenes sieve to the next, from $\pgap(\pml{p_k})$ to 
$\pgap(\pml{p_{k+1}})$.  This structure is sufficient to enable us to give exact counts for gaps and for 
sufficiently short constellations in $\pgap(\pml{p})$ across all stages of the sieve.

\subsection{Some conjectures and open problems regarding gaps between primes.}
Open problems regarding gaps and constellations between prime numbers include
the following.
\begin{itemize}
\item {\em Twin Prime Conjecture} - There are infinitely many pairs of consecutive primes with
gap $g=2$.
\item {\em Polignac's Conjecture} - For every even number $2n$, there are infinitely many pairs of 
consecutive primes with gap $g=2n$.
\item {\em Primorial conjecture} - The gap $g=6=\pml{3}$ occurs more often than the gap
$g=2$, and eventually the gap $g=30=\pml{5}$ occurs more often than the gap $g=6$.
\item {\em HL Conjecture B} - From page 42 of Hardy and Littlewood \cite{HL}:  for any even $k$, the number of prime pairs $q$ and $q+k$ such that $q+k < n$ is approximately
$$ 2 C_2 \frac{n}{(\log n)^2} \prod_{p \neq 2, \; p | k} \frac{p-1}{p-2}.$$
\item {\em ET Spikes} - From p.377 of Erd\"os and Tur\'an \cite{ET}, that it is very probable that
$$ \lim \sup \frac{g_{k+1}}{g_k} = \infty \gap {\rm and} \gap \lim \inf \frac{g_{k+1}}{g_k} = 0.$$
\item {\em ET Superlinearity} - On p.378 of Erd\"os and Tur\'an \cite{ET}, the open question is posed
whether for every $k > 1$ there are infinitely many $n$ such that
$$ g_n < g_{n+1} < g_{n+2} < \cdots < g_{n+k} .$$
\end{itemize}

These problems and others regarding the gaps and differences among primes are
usually approached through sophisticated probabilistic models, rooted in the prime number theorem.
Seminal works for our studies include \cite{HL, HW, Cramer}.  
Several estimates on gaps derived from these models
have been corroborated computationally.  These computations have
addressed the occurrence of twin primes \cite{Brent3, NicelyTwins, PSZ, IJ, JLB},
and some have corroborated the estimates in Conjecture B for other gaps
\cite{Brent, GranRaces}.

Work on specific constellations among primes includes the study of prime quadruplets
\cite{HL, quads},
which corresponds to the constellation $2,4,2$.  This is two pairs of twin primes separated
by a gap of $4$, the densest possible occurrence of primes in the large. 
The estimates for prime quadruplets have also been supported computationally
\cite{NicelyQuads}.

\subsection{Analogues demonstrated for Eratosthenes sieve.}
We do not resolve any of the open problems as stated above for gaps between primes.
However, we are able to resolve their analogues for gaps in the stages of Eratosthenes sieve.
Through our work below we prove the following.
\begin{itemize}
\item {\em Twin Generators} - The number of gaps $g=2$ in the cycle of gaps
$\pgap(\pml{p_k})$ is
$$N_2(\pml{p_k}) = \prod_{q=3}^{p_k} (q-2).$$
\item {\em Polignac's Conjecture and HL Conjecture B} - For every even number $2n$, 
the gap $g=2n$ arises
at some stage in the cycle of gaps $\pgap(\pml{p})$ and thereafter the ratio of its
occurrences in a cycle of gaps to the number of gaps $g=2$ approaches the asymptotic
value suggested by Hardy and Littlewood's Conjecture B
$$ \lim_{p \fto \infty} \frac{N_{2n}(\pml{p})}{N_2(\pml{p})} = \prod_{q>2, \; q | n} \frac{q-1}{q-2}.$$
\item {\em Primorial conjecture} - The dynamic system that yields the preceding result 
tells us that for primorial gaps $g= \pml{p_{k-1}}$ and $g=\pml{p_k}$,
$$ \lim_{q \fto \infty} \frac{N_{\pml{p_k}}(\pml{q})}{N_{\pml{p_{k-1}}}(\pml{q})} = 
 \frac{p_k-1}{p_k-2}.$$
The eigenvalues of the dynamic system indicate how quickly the values will
converge to the asymptotic ratio.
\item {\em ET Spikes} - For gaps in the cycles of gaps,
$$ \lim \sup \frac{g_{k+1}}{g_k} = \infty \gap {\rm and} \gap \lim \inf \frac{g_{k+1}}{g_k} = 0.$$
In particular, for $g_k=2$, the adjacent gaps $g_{k-1}$ and $g_{k+1}$ become arbitrarily large in 
later stages of the sieve.
\item {\em ET Superlinear growth} - For every $k > 1$ there is a cycle of gaps $\pgap(\pml{p})$
with a constellation of $k$ consecutive gaps such that
$$  g_{n+1} < g_{n+2} < \cdots < g_{n+k}.$$
This constellation persists across all subsequent stages of the sieve, and its population increases
by the factor $p-k-1$ at each stage.
\item {\em ET Superlinear decay} - For every $k > 1$ there is a cycle of gaps $\pgap(\pml{p})$
with a constellation of $k$ consecutive gaps such that
$$ g_{n+1} >  g_{n+2} > \cdots  > g_{n+k}.$$
This constellation persists across all subsequent stages of the sieve, and its population increases
by the factor $p-k-1$ at each stage.
\end{itemize}

These results are deterministic, not probabilistic.  We develop a population model below
that describes the growth of the populations of various gaps in the cycle of gaps, across
the stages of Eratosthenes sieve.

All gaps between prime numbers arise in a cycle of gaps.  To connect our results 
to the desired results on gaps between primes, we need to better understand how
gaps survive later stages of the sieve, to be affirmed as gaps between
primes.

\section{The cycle of gaps}
After the first two stages of Eratosthenes sieve, we have removed the multiples of $2$ and $3$.
The candidate primes at this stage of the sieve are 
$$(1),5,7,11,13,17,19,23,25,29,31,35,37,41,43,\ldots$$
We investigate the structure of these sequences of candidate primes
by studying the cycle of gaps in the fundamental cycle.

For example, for the candidate primes listed above, the first gap from $1$ to $5$ is $g=4$,
the second gap from $5$ to $7$ is $g=2$, then $g=4$ from $7$ to $11$, and so on.
The {\em cycle of gaps} $\pgap(\pml{3})$ is $42$.  To reduce visual clutter, we write
the cycles of gaps as a concatenation of single digit gaps, reserving the use of commas to 
delineate gaps of two or more digits.

$$\pgap(\pml{3}) = 42, \gap {\rm with} \gap g_{3,1}=4 \gap {\rm and} \gap 
g_{3,2}=2.$$

Advancing Eratosthenes sieve one more stage, we identify $5$ as the next prime 
and remove the multiples of $5$ from the list of candidates, leaving us with
$$(1),7,11,13,17,19,23,29,31,37,41,43,47,49,53,59,61, \ldots$$
We calculate the cycle of gaps at this stage to be
$\pgap(\pml{5}) = 64242462.$

We note that $\pgap(\pml{p})$ consists of $\phi(\pml{p})$ gaps that sum to $\pml{p}$.

\subsection{Recursion on the cycle of gaps}
There is a nice recursion which produces $\pgap(p_{k+1})$ directly from
$\pgap(p_k)$.  We concatenate $p_{k+1}$ copies of $\pgap(p_k)$, and
add together certain gaps as indicated by the entry-wise product
$p_{k+1}*\pgap(p_k)$.  

\begin{lemma} \label{RecursLemma}
The cycle of gaps $\pgap(\pml{p_{k+1}})$ is derived recursively from 
$\pgap(\pml{p_k})$.
Each stage in the recursion consists of the following three steps:
\begin{itemize}
\item[R1.] Determine the next prime, $p_{k+1} = g_{k,1} + 1$.
\item[R2.] Concatenate $p_{k+1}$ copies of $\pgap(\pml{p_k})$.
\item[R3.] Add adjacent gaps as indicated by the elementwise product 
$p_{k+1}*\pgap(\pml{p_k})$:  let $i_1=1$ and add together $g_{i_1}+g_{i_1+1}$; then for 
$n=1,\ldots,\phi(N)$, add $g_{j}+g_{j+1}$ and let 
$i_{n+1}=j$ if the running sum of the concatenated gaps from $g_{i_n}$ to $g_j$ is
$p_{k+1}*g_{n}.$
\end{itemize}
\end{lemma}

\begin{proof}  
Let $\pgap(\pml{p_k})$ be the cycle of gaps for the stage of Eratosthenes sieve
after the multiples of the primes up through $p_k$ have been removed.  We show that the recursion
R1-R2-R3 on $\pgap(\pml{p_k})$ produces the cycle of gaps for the next stage, corresponding
to the removal of multiples of $p_{k+1}$.

There is a natural one-to-one correspondence between the gaps in the cycle
of gaps $\pgap(\pml{p_k})$ and the generators of $\Z \bmod \pml{p_k}$.
For $j=1,\ldots, \phi(\pml{p_k})$ let
\begin{equation}\label{EqGen}
\gamma_{k,j} = 1 + \sum_{i=1}^j g_i.
\end{equation}
These $\gamma_{k,j}$ are the generators in $\Z \bmod \pml{p_k}$, with 
$\gamma_{k,\phi(\pml{p_k})} \equiv 1 \bmod \pml{p_k}$.
 
The  $\ord{j}$ candidate prime at this stage of the sieve is given by $\gamma_{k,j}$.

The next prime $p_{k+1}$ will be $\gamma_{k,1}$, since this will be the smallest
integer both greater than $1$ and coprime to $\pml{p_k}$.

The second step of the recursion extends our list of possible primes
up to $\pml{p_{k+1}}+1$, the reach of the fundamental cycle for $\pml{p_{k+1}}$.
For the gaps $g_j$ we extend the indexing on $j$ to cover these 
concatenated copies.  These $p_{k+1}$ concatenated copies of $\pgap(\pml{p_k})$
correspond to all the numbers from $1$ to $\pml{p_{k+1}}+1$ which are coprime
to $\pml{p_k}$.  For the set of generators of $\pml{p_{k+1}}$, we need only remove
the multiples of $p_{k+1}$. 

The third step removes the multiples of $p_{k+1}$.
Removing a possible prime amounts to
adding together the gaps on either side of this entry.
The only multiples of $p_{k+1}$ which remain in the copies of $\pgap(\pml{p_k})$
are those multiples all of whose prime factors are greater than $p_k$.
After $p_{k+1}$ itself, the next multiple to be removed will be $p_{k+1}^2$.

The multiples we seek to remove are given by $p_{k+1}$ times the generators
of $\Z \bmod \pml{p_k}$.  The consecutive differences between these will be given
by $p_{k+1} * g_j$, and the sequence $p_{k+1}*\pgap(\pml{p_k})$ suffices to cover
the concatenated copies of $\pgap(\pml{p_k})$.  We need not consider any fewer nor any
more multiples of $p_{k+1}$ to obtain the generators for $\pgap(\pml{p_{k+1}})$.

In the statement of R3, the index $n$ moves through the copy of
$\pgap(\pml{p_k})$ being multiplied by $p_{k+1}$, and the indices $\tilde{i}_n$
mark the index $j$ at which the addition of gaps is to occur.
The multiples of $p_{k+1}$ in the sieve up through $p_k$ are given by
$p_{k+1}$ itself and $p_{k+1}*\gamma_{k,j}$ for $j=1,\ldots,\phi(\pml{p_k})$.
The difference between successive multiples is $p_{k+1}*g_j$.
\end{proof}

%
\begin{figure}[t]
\begin{center}
\includegraphics[width=5in]{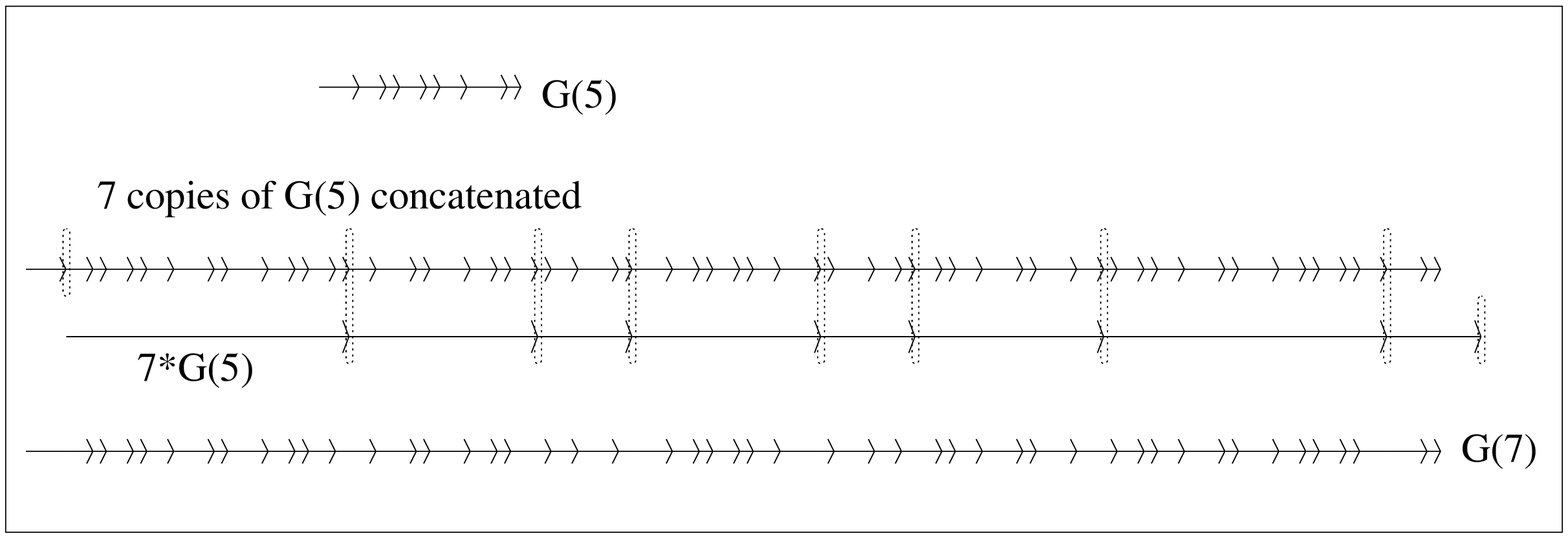}
\caption{\label{RecursFig} Illustrating the recursion that
produces the gaps for the next stage of Eratosthenes sieve.
The cycle of gaps $\pgap(\pml{7})$ is produced from $\pgap(\pml{5})$ by
concatenating $7$ copies, then adding the gaps indicated by
the element-wise product $7*\pgap(\pml{5})$.}
\label{default}
\end{center}
\end{figure}

We call the additions in step R3 the {\em closure} of the two adjacent gaps.  

The first closure in step R3 corresponds to noting the next prime number $p_{k+1}$.
The remaining closures in step R3 correspond to removing from the candidate 
primes the composite numbers whose smallest prime factor is $p_{k+1}$.
From step R2, the candidate primes have the form
 $\gamma + j\cdot \pml{p_k}$, for a generator $\gamma$ of $\Z \bmod \pml{p_k}$.

\noindent {\bf Example: $\pgap(5)$.}
We start with $\pgap(3)=42.$
\begin{itemize}
\item[R1.]$p_{k+1}=5.$
\item[R2.] Concatenate five copies of $\pgap(3)$:
$$4242424242.$$
\item[R3.] Add together the gaps after the initial gap $g=4$ and 
thereafter after cumulative differences of $5*\pgap(3)=20,10$:\\
\begin{tabular}{rcc}
$\pgap(5)$ & $=$ & $4+\overbrace{2424242}^{20}+\overbrace{42\;\;}^{10}$ \\
 & $=$ & $64242462.$ \\
\end{tabular} \\
Note that the last addition wraps around the end of the cycle
and recloses the gap after the first $4$.
\end{itemize}

\noindent{\bf Example: $\pgap(\pml{7})$.}
As a second example of the recursion,
we construct $\pgap(\pml{7})$ from $\pgap(\pml{5})=64242462$.
Figure~\ref{RecursFig} provides an illustration of this construction.

\begin{itemize}
\item[R1.] Identify the next prime, $p_{k+1}= g_1+1 = 7.$
\item[R2.] Concatenate seven copies of $\pgap(\pml{5})$:
$$\scriptstyle 64242462 \; 64242462 \; 64242462 \; 64242462\; 64242462 \; 64242462 \;64242462$$
\item[R3.] Add together the gaps after the leading $6$ and 
thereafter after differences of 
$ 7*\pgap(\pml{5}) = 42, 28, 14, 28, 14, 28, 42, 14 $:
$$\begin{array}{l} 
\pgap(\pml{7}) \; = \\
{\scriptstyle
  6+\overbrace{\scriptstyle 424246264242}^{42}+
 \overbrace{\scriptstyle 4626424}^{28}+\overbrace{\scriptstyle 2462}^{14}+
 \overbrace{\scriptstyle 6424246}^{28}+\overbrace{\scriptstyle 2642}^{14}+
 \overbrace{\scriptstyle 4246264}^{28}+\overbrace{\scriptstyle 242462642424}^{42}+
 \overbrace{\scriptstyle 62 \; }^{14}} \\
=  \; 
 {\it 10}, 2424626424 {\it 6} 62642 {\it 6} 46 {\it 8} 42424 {\it 8} 64 {\it 6} 24626 {\it 6} 
  4246264242, {\it 10}, 2
\end{array}
$$
The final difference of $14$ wraps around the end of the cycle,
 from the addition preceding the final $6$ to the 
addition after the first $6$.
\end{itemize}

\begin{remark}\label{EasyRmk}
The following results are easily established for $\pgap(p_k)$:

\begin{enumerate}
\item The cycle of gaps $\pgap(\pml{p_k})$ consists of $\phi(\pml{p_k})$ gaps
that sum to $\pml{p_k}$.
\item The first difference between closures is $p_{k+1}*(p_{k+1}-1)$,
which removes $p_{k+1}^2$ from the list of candidate primes.
\item The last entry in $\pgap(\pml{p_k})$ is always $2$. 
This difference goes from $-1$ to $+1$ in $\Zmod$.
\item The last difference $p_{k+1}*2$ between closures in step R3, wraps from
$-p_{k+1}$ to $p_{k+1}$ in $\Zmodp$.
\item Except for the final $2$, the cycle of differences is symmetric:
$g_{k,j}=g_{k,\phi(\pml{p_k})-j}$.
\item If $m+1$ consecutive gaps have the same value,
$$g_{k,j}=g_{k,j+1}= \cdots = g_{k,j+m}=g,$$ 
then $g = 0 \bmod p$ for all primes $p \leq m+2$.  Note that this constellation 
corresponds to $m+2$ consecutive primes in arithmetic progression.
\item The middle of the cycle $\pgap(\pml{p_k})$ is the sequence 
$$2^j,2^{j-1},\ldots,42424,\ldots,2^{j-1},2^j$$
in which $j$ is the smallest number such that $2^{j+1}>p_{k+1}$.
\end{enumerate}
\end{remark}

There is an interesting fractal character to the recursion.  To produce the next cycle of gaps
$\pgap(\pml{p_{k+1}})$ from the current one, $\pgap(\pml{p_k})$, we concatenate
$p_{k+1}$ copies of the current cycle, take an expanded copy of the current cycle, and close gaps
as indicated by that expanded copy.  In the discrete dynamic system that we develop below, we 
don't believe that all of the power in this self-similarity has yet been captured.

\subsection{Every possible closure of adjacent gaps occurs exactly once}

\begin{theorem}\label{DelThm}
Each possible closure of adjacent gaps in the cycle $\pgap(\pml{p_k})$
occurs exactly once in the recursive construction of $\pgap(\pml{p_{k+1}}).$
\end{theorem}

\begin{proof}
This is an implication of the Chinese Remainder Theorem.
Each entry in $\pgap(\pml{p_k})$ corresponds to one of the generators
of $\Zmod$. The first gap $g_1$ corresponds to $p_{k+1}$, and
thereafter $g_j$ corresponds to $\gamma_{k,j} = 1+\sum_{i=1}^j g_i$.
These correspond in turn to unique combinations of nonzero
residues modulo the primes $2,3,\ldots,p_k$.  

In step R2, we concatenate $p_{k+1}$
copies of $\pgap(\pml{p_k})$.  For each gap $g_j$ in $\pgap(\pml{p_k})$
there are $p_{k+1}$ copies of this gap after step R2, corresponding to 
$$\gamma_{k,j} + i \cdot \pml{p_k} \biggap {\rm for} \gap i = 0, \ldots, p_{k+1}-1.$$
For each copy, 
the combination of residues for $\gamma_{k,j}$ modulo $2,3,\ldots,p_k$
is augmented by a unique residue modulo
$p_{k+1}$.  Exactly one of these has residue $0 \bmod p_{k+1}$, so we
perform $g_{j}+g_{j+1}$ for this copy and only this copy of $g_j$.
\end{proof}

\begin{corollary}
In $\pgap(\pml{p_{k+1}})$ there are at least two gaps of size $g=2p_k$.
\end{corollary}

\begin{proof}
In forming $\pgap(\pml{p_{k}})$, in step R2 we concatenate $p_{k}$ copies of
$\pgap(\pml{p_{k-1}})$.  Each copy of $\pgap(\pml{p_{k-1}})$ begins
with the gap $g=p_k-1$ and ends with the constellation $(p_k-1) 2$.
At the transition between copies we have the
sequence $(p_k-1) 2 (p_k-1)$.  In step R3 of forming $\pgap(\pml{p_{k}})$ each of
the two closures takes place, producing the constellations $(p_k-1)(p_k+1)$ and
$(p_k+1)(p_k-1)$ in $\pgap(\pml{p_k})$.  
In forming $\pgap(\pml{p_{k+1}})$ exactly one of the $p_{k+1}$ copies of each of
these constellations is closed, to create two gaps in $\pgap(\pml{p_{k+1}})$ of
size $2p_k$.
\end{proof}

By the symmetry of $\pgap(\pml{p})$ and by the symmetry of the locations of the
closures in step R3, we note that the two gaps $g=2p_k$ are located symmetrically to 
each other in $\pgap(\pml{p_{k+1}})$.

\section{Enumerating gaps, constellations, and driving terms}

By analyzing the application of Theorem~\ref{DelThm} to the recursion, we can derive
exact counts of the occurrences of specific gaps and specific constellations across all
stages of the sieve.  

We start by exploring a few motivating examples, after which we 
describe the general process as a discrete dynamic system -- a population model with initial
conditions and driving terms.  Fortunately, although
the transfer matrix $M_J(p_k)$ for this dynamic system depends on the prime $p_k$, its eigenstructure
is beautifully simple, enabling us to provide correspondingly simple descriptions of the
asymptotic behavior of the populations.  In this setting, the populations are the numbers
of occurrences of specific gaps or constellations across stages of Eratosthenes sieve.

\subsection{Motivating examples}
We start with the cycle of gaps 
$$\pgap(\pml{5})=64242462$$
and study the persistence of its gaps and constellations through later stages of the sieve.

Using the notation $N_s(\pml{p})$ to denote the number of occurrences of the constellation $s$
in the cycle of gaps $\pgap(\pml{p})$, we identify some initial conditions:
\begin{center}
\begin{tabular}{rrrr}
$N_2(\pml{5}) = 3$ & $N_{24}(\pml{5}) = 2$ & $N_{242}(\pml{5}) = 1$ & 
 $N_{42424}(\pml{5}) = 1$ \\
 $N_4(\pml{5}) = 3$ & $N_{42}(\pml{5}) = 2$ & $N_{424}(\pml{5}) = 2$ & 
  \\
 $N_6(\pml{5}) = 2$ & $N_{62}(\pml{5}) = 1$ & $N_{626}(\pml{5}) = 1$ & 
\end{tabular} 
\end{center}

{\bf Enumerating the gaps $g=2$ and $g=4$.}  For the gap $g=2$, we start with
$N_2(\pml{5})=3$.  In forming $\pgap(\pml{7})$, in step R2 we create $7$ copies
of each of the three $2$'s, and from each family of seven copies, in step R3 we lose two of these
seven copies -- one for the closure to the left, and another for the closure to the right.

Could the two closures occur on the same copy of a $2$?  We observe that in step R3, the
distances between closures is governed by the entries in $7*\pgap(\pml{5})$, so the
minimum distance between closures in forming $\pgap(\pml{7})$ is $7*2 = 14$.
Thus the two closures cannot occur on the same copy of a $2$.

So for each $g=2$ in $\pgap(\pml{5})$, in step R2 we create seven copies, and in step R3 we
close two of these seven copies, one from the left and one from the right.  Noting that closures can
contribute new gaps, we observe from $\pgap(\pml{5})$ that no $2$'s or $4$'s will be created through closures.  So the populations of the gaps $g=2$ and $g=4$ are completely described by:
\begin{equation}\label{Eq2System}
\begin{array}{lcl}
N_2(\pml{p_{k+1}}) = (p_{k+1}-2) \cdot N_2(\pml{p_k}) & {\rm with} & 
 N_2(\pml{5}) = 3 \\
N_4(\pml{p_{k+1}}) = (p_{k+1}-2) \cdot N_4(\pml{p_k}) & {\rm with} & 
 N_4(\pml{5}) = 3
\end{array}
\end{equation}
From this we see immediately that at every stage of Eratosthenes sieve, 
$N_2(\pml{p}) = N_4(\pml{p})$, and that the number of gaps $g=2$ in the
cycle of gaps $\pgap(\pml{p})$, denoted $N_2(\pml{p})$, grows superexponentially by
a factor of $p-2$ as we increase the size of the prime $p$ through the stages of the sieve.

{\bf Enumerating the gaps $g=6$ and its driving terms.}
For the gap $g=6$, we count $N_6(\pml{5})=2$.  In forming $\pgap(\pml{7})$, 
we will create seven copies of each of these gaps and close two of the copies for each initial gap.
However, we will also gain gaps $g=6$ from the closures of the constellations $s=24$ and $s=42$.

We call these constellations $s=24$ and $s=42$ {\em driving terms} for the gap $g=6$.  These driving
terms are of length $2$.  We observe that these constellations do not themselves have driving terms.
For $s=24$, we initially have $N_{24}(\pml{5})=2$, and under the recursion that forms
$\pgap(\pml{7})$, we create seven copies of each constellation $s=24$, and we close 
{\em three} of these copies.  The left and right closures remove these copies from the system for
$g=6$, and the middle closure produces a gap $g=6$. 

We can express the system for the population of gaps $g=6$ as:
\begin{equation}\label{Eq6System}
\left[ \begin{array}{c} N_6 \\  N_{24}+N_{42} \end{array}\right]_{\pml{p_{k+1}}}
= \left[\begin{array}{cc} p_{k+1}-2 & 1 \\ 0 & p_{k+1}-3 \end{array}\right]
\left[ \begin{array}{c} N_6 \\ N_{24} + N_{42} \end{array}\right]_{\pml{p_k}}
\end{equation}
with $N_6(\pml{5})=2$ and $N_{24}(\pml{5})=N_{42}(\pml{5})=2$.

By the symmetry of $\pgap(\pml{p})$, we know that 
$N_{24}(\pml{p})=N_{42}(\pml{p})$ for all $p$, but the above approach of recording this
as an addition will help us develop the general form for the dynamic system.

How does the population of gaps $g=6$ compare to that for $g=2$?  In $\pgap(\pml{7})$, 
there are still more gaps $2$ than $6$'s:
$$
 N_6(\pml{7})  =  5\cdot 2 + (2+2) = 14 <  N_2(\pml{7}) = 5\cdot 3 = 15.
$$
There are now $16$ driving terms for $6$:
$ (N_{24} + N_{42})(\pml{7})  =  4 \cdot (2+2)  = 16,$ which help make $6$'s more numerous than
$2$'s in $\pgap(\pml{11})$:
$$
N_6(\pml{11}) =  9\cdot 14 + (16) = 142  >  N_2(\pml{11}) = 9\cdot 15 = 135.
$$
 Thereafter, both populations $N_2(\pml{p})$ and $N_6(\pml{p})$
 are growing by the factor $(p-2)$, and the gap $g=6$ has 
driving terms whose populations grow by the factor $(p-3)$.

{\bf Enumerating the gaps $g=8$ and its driving terms.}
For the gap $g=8$, we have $N_8(\pml{5})=0$; however, there are driving terms of length two $s=26$ and $s=62$, and in this case there is a driving term of length three $s=242$.  No other constellations
in $\pgap(\pml{5})$ sum to $8$.  So how will the population of
the gap $g=8$ evolve over stages of the sieve?

As we have seen with the gaps $g=2, \; 4, \; 6$, in forming $\pgap(\pml{p_{k+1}})$ each gap $g=8$ will initially generate $p_{k+1}$ copies in step R2 of which $p_{k+1}-2$ will survive step R3.  Each instance
of $s=26$ or $s=62$ will generate one additional gap $g=8$ upon the interior closure, two copies will
be lost from the exterior closures, and $p_{k+1}-3$ copies will survive step R3 of the recursion.

The driving term of length three, $s=242$, will add to the populations of the driving terms
of length $2$.
In forming $\pgap(\pml{7})$, we will
create seven copies of $s=242$ in step R2.  In step R3, for the seven copies of $s=242$, 
the two exterior closures increase the sum,
removing the resulting constellation as a driving term for $g=8$; the two interior closures create
driving terms of length two ($s=62$ and $s=26$), and three copies of $s=242$ survive intact.

We now state this action as a general lemma for any constellation, which 
includes gaps as constellations of length one.

\begin{lemma}\label{Lemma2p}
For $p_k \ge 3$, let $s$ be a constellation of sum $g$ and length $j$, such that $g < 2\cdot p_{k+1}$.
Then for each instance of $s$ in $\pgap(\pml{p_k})$, in forming $\pgap(\pml{p_{k+1}})$, 
in step R2 we create $p_{k+1}$ copies of this instance of $s$, and the $j+1$ closures in step R3 
occur in distinct copies.  

Thus, under the recursion at this stage of the sieve,
each instance of $s$ in $\pgap(\pml{p_k})$generates $p_{k+1}-j-1$ copies of $s$ in 
$\pgap(\pml{p_{k+1}})$; the interior closures generate $j-1$ constellations of sum $g$ and
length $j-1$; and the two exterior closures increase the sum of the resulting constellation in two 
distinct copies, removing these from being driving terms for the gap $g$.
\end{lemma}

The proof is a straightforward application of Theorem~\ref{DelThm}, 
but we do want to emphasize the role that the condition $g < 2\cdot p_{k+1}$
plays.  In step R3 of the recursion, as we perform closures across the $p_{k+1}$ concatenated copies
of $\pgap(\pml{p_k})$, the distances between the closures is given by the elementwise product
$p_{k+1} * \pgap(\pml{p_k})$.  Since the minimum gap in $\pgap(\pml{p})$ is $2$, the minimum 
distance between closures is $2\cdot p_{k+1}$.  And the condition $g < 2\cdot p_{k+1}$ ensures 
that the closures will
therefore occur in distinct copies of any instance of the constellation in $\pgap(\pml{p_k})$
created in step R2.

The count in Lemma~\ref{Lemma2p} is scoped to instances of a constellation.
These instances may overlap, but the count still holds.  For example, the gap $g=10$ has a driving
term $s=424$ of length three.  In $\pgap(\pml{5})= 64242462$, the two occurrences of $s=424$
overlap on a $4$.  The exterior closure for one is an interior closure for the other.  The count given
in the lemma tracks these automatically.

We illustrate Lemma~\ref{Lemma2p} in Figure~\ref{SystemFig}.  

\begin{figure}[t]
\centering
\includegraphics[width=4.875in]{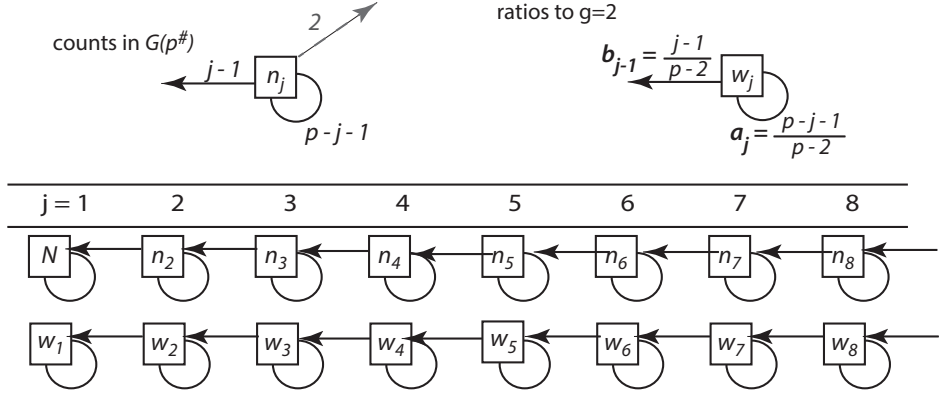}
\caption{\label{SystemFig} This figure illustrates the dynamic system of
Lemma~\ref{Lemma2p} through stages of the recursion for
$\pgap(\pml{p})$.
The coefficients of the system at each stage of the recursion
are independent of the specific gap and its driving terms.  
We illustrate the system for the recursive counts
$N$ and $n_j$ for a gap and its driving terms.  Since the raw counts are superexponential, we take
the ratio $w_j$ of the count for each constellation to the simplest counts $\N{p}{2}=\N{p}{4}$. }
\end{figure}

A direct corollary to Theorem~\ref{DelThm} and Lemma~\ref{Lemma2p} gives us an
exact description of the growth of the populations of various constellations across all stages
of Eratosthenes sieve.  (Keep in mind that a gap is a constellation of length $1$.)

\begin{corollary}\label{CorGrowth}
If $s$ is any constellation in $\pgap(\pml{p_k})$ of length $j$ and sum $g < 2p_{k+1}$,
with $n_{s,j+1}(\pml{p_k})$ driving terms of length $j+1$ in $\pgap(\pml{p_k})$,
then 
$$N_s(\pml{p_{k+1}}) = (p_{k+1}-j-1)\cdot N_s(\pml{p_k}) + 1\cdot n_{s,j+1}(\pml{p_k}).$$
\end{corollary}

From this corollary, we note that the coefficients for the population model do not depend on the
constellation $s$.  The first-order growth of the population of every constellation $s$ of 
length $j$ and sum $g < 2p_{k+1}$ is given by the factor $p_{k+1}-j-1$.  This is independent
of the sequence of gaps within $s$.  As we have seen above, constellations may differ
significantly in the populations of their driving terms.

Although the asymptotic growth of all constellations of length $j$
is equal, the initial conditions and driving terms are important.
Brent \cite{Brent} made analogous observations for single gaps ($j=1$).
His Table 2 indicates the importance of the lower-order effects in
estimating relative occurrences of certain gaps.

\subsection{Relative populations of $g=2, \; 6, \; 8, \; 10$}
What can we say about the relative populations of the gaps $g=2, \; 6, \; 8$ over later stages of the
sieve?  The population of every gap grows by a factor of $p-2$.  The populations differ by the presence 
of driving terms of various lengths and by the initial conditions.

We proceed by normalizing each population by the population of the gap $g=2$.  
To compare the populations of any gap $g$ to the gap $2$ over later stages of the sieve, 
we take the ratio
\begin{equation}\label{EqDefW}
\Rat{1}{g}(\pml{p}) = \frac{ N_g(\pml{p})}{N_2(\pml{p})}.
\end{equation}
Letting $n_{g,j}(\pml{p})$ denote the number of all driving terms of sum $g$ and length $j$
in the cycle of gaps $\pgap(\pml{p})$, we can extend this definition to
\begin{equation*}
\Rat{j}{g}(\pml{p}) = \frac{ n_{g,j}(\pml{p})}{N_2(\pml{p})}.
\end{equation*}

\renewcommand\arraystretch{2}
These ratios for the gaps $g=6$ and $g=8$ are given by the $3$-dimensional dynamic system:
\begin{eqnarray*}
\left[\begin{array}{c} \Rat{1}{g} \\ \Rat{2}{g} \\ \Rat{3}{g} \end{array} \right] _{\pml{p_{k+1}}} & = &
\left[\begin{array}{ccc} 1 & \frac{1}{p_{k+1}-2} & 0 \\ 0 & \frac{p_{k+1}-3}{p_{k+1}-2} & \frac{2}{p_{k+1}-2} \\
 0 & 0 & \frac{p_{k+1}-4}{p_{k+1}-2} \end{array} \right] \cdot
\left[\begin{array}{c} \Rat{1}{g} \\ \Rat{2}{g} \\ \Rat{3}{g} \end{array} \right]_{\pml{p_k}} \\
{\rm or} & & \\
 \wgp{g}{\pml{p_{k+1}}} & = &
   \MJP{3}{p_{k+1}} \cdot \wgp{g}{\pml{p_k}} 
\end{eqnarray*}
\renewcommand\arraystretch{1}
with initial conditions
\begin{equation*}
\begin{array}{ccc}
\wgp{6}{\pml{5}}  = \left[ \begin{array}{c} 2/3 \\ 4/3 \\ 0 \end{array}\right], & 
\wgp{8}{\pml{5}} = \left[ \begin{array}{c} 0 \\ 2/3 \\ 1/3 \end{array}\right] 
\end{array}
\end{equation*}

For this dynamic system, our attention turns to the $3 \times 3$ system matrix $M_3(p)$ and its
eigenstructure.
Here the system matrix depends on the prime $p$ (but not on the gap $g$), so that as we iterate, 
we have to keep track of this dependence.
$$
\wgp{g}{\pml{p_k}}  =  \MJP{3}{p_k} \cdot \MJP{3}{p_{k-1}} \cdots \MJP{3}{p_1} \wgp{g}{\pml{p_0}} 
$$

A simple calculation shows that we are in luck.  The eigenvalues of $M_3(p)$ depend on $p$
but the eigenvectors do not.  We write the eigenstructure of $M_3(p)$ as
$$ \MJP{3}{p} =  R_3 \cdot \left. \Lambda_3 \right|_p \cdot L_3$$
in which $\Lambda(p)$ is the diagonal matrix of eigenvalues, $R$ is the matrix of right eigenvectors,
and $L$ is the matrix of left eigenvectors, such that $L\cdot R = I$.

If it is true that the eigenvectors do not depend on $p$, then the iterative system simplifies:
\begin{eqnarray*}
\wgp{g}{\pml{p_k}} &  = & \MJP{3}{p_k} \cdot \MJP{3}{p_{k-1}} \cdots \MJP{3}{p_1} \wgp{g}{\pml{p_0}} \\
  & = & R \cdot \AJP{3}{p_k} \cdot \AJP{3}{p_{k-1}} \cdots \AJP{3}{p_1} \cdot L \wgp{g}{\pml{p_0}} 
\end{eqnarray*}
Here the dependence on $p$ leads to a product of diagonal matrices.  We exhibit the full
eigenstructure for dimension $3$, to confirm that the eigenvectors do not depend on $p$,
so that we can complete the calculations.
\begin{equation*}
\begin{array}{rcccc} \MJP{3}{p} & = & R_3 & \AJP{3}{p} & L_3 \\
 & = & \left[\begin{array}{rrr} 1 & -1 & 1 \\ 0 & 1 & -2 \\ 0 & 0 & 1 \end{array}\right]
 \cdot & \left[\begin{array}{rcc} 1 & 0 & 0 \\ 0 & \frac{p-3}{p-2} & 0 \\ 0 & 0 & \frac{p-4}{p-2} \end{array}\right] 
  & \cdot \left[\begin{array}{rrr} 1 & 1 & 1 \\ 0 & 1 & 2 \\ 0 & 0 & 1 \end{array}\right]
  \end{array}
 \end{equation*}
 
If we fix $p_0$, then we can define $M^k_3 = M_3(p_k) \cdots M_3(p_1)$:
\begin{equation*}
\begin{array}{rcccc} M^k_3 & = & R_3 & \AJP{3}{p_k} \cdots \AJP{3}{p_1} & L_3 \\
 & = & \left[\begin{array}{rrr} 1 & -1 & 1 \\ 0 & 1 & -2 \\ 0 & 0 & 1 \end{array}\right]
 \cdot & \left[\begin{array}{rcc} 1 & 0 & 0 \\ 0 & \prod_{p_1}^{p_k} \frac{p-3}{p-2} & 0 \\ 
   0 & 0 & \prod_{p_1}^{p_k} \frac{p-4}{p-2} \end{array}\right] 
  & \cdot \left[\begin{array}{rrr} 1 & 1 & 1 \\ 0 & 1 & 2 \\ 0 & 0 & 1 \end{array}\right]
  \end{array}
 \end{equation*}
 
 Let $a_2^k = \prod_{p_1}^{p_k} \frac{p-3}{p-2}$, and $a_3^k = \prod_{p_1}^{p_k} \frac{p-4}{p-2}$.
 
 So for any gap $g < 2p_1$ that has driving terms of a maximum length of $3$, once we know the 
 initial populations in $\pgap(\pml{p_0})$, we can use the eigenstructure
 of $M^k_3$ to completely characterize the populations of $g$ and its driving terms in
 a very compact form.  Starting with the initial conditions
 $$
\wgp{g}{\pml{p_0}}  = \left[ \begin{array}{c} 
 w_{g,1} \\ w_{g,2} \\ w_{g,3} \end{array} \right]_{\pml{p_0}},
 $$
 we apply the left eigenvectors $L_3$ to obtain the coordinates relative to the basis
 of right eigenvectors $R_3$.  After this transformation, we can apply the actions of 
 the eigenvalues $\Lambda_3(p)$ directly to the individual right eigenvectors.  
\begin{eqnarray*}
\wgp{g}{\pml{p_k}} & = & M^k_3 \cdot \wgp{g}{\pml{p_0}} \\
 & = & R_3 \cdot \AJP{3}{p_k} \cdots \AJP{3}{p_1} \cdot L_3 \cdot \wgp{g}{p_0} \\
 & = & \left[\begin{array}{rrr} 1 & -1 & 1 \\ 0 & 1 & -2 \\ 0 & 0 & 1 \end{array}\right]
 \cdot  \left[\begin{array}{rrr} 1 & 0 & 0 \\ 0 & a_2^k & 0 \\ 
   0 & 0 & a_3^k \end{array}\right] 
  \cdot \left[\begin{array}{rrr} 1 & 1 & 1 \\ 0 & 1 & 2 \\ 0 & 0 & 1 \end{array}\right] \cdot 
  \left[ \begin{array}{c} w_{g,1} \\ w_{g,2} \\ w_{g,3} \end{array} \right] _{\pml{p_0}}\\
  & = & \left[\begin{array}{rrr} 1 & -1 & 1 \\ 0 & 1 & -2 \\ 0 & 0 & 1 \end{array}\right]
 \cdot  \left[\begin{array}{rrr} 1 & 0 & 0 \\ 0 & a_2^k & 0 \\ 
   0 & 0 & a_3^k \end{array}\right] 
  \cdot  \left[ \begin{array}{r} w_{g,1}+ w_{g,2}+w_{g,3} \\ 
   w_{g,2}+2 w_{g,3} \\ w_{g,3} \end{array} \right]_{\pml{p_0}} \\
   & = & \left[\begin{array}{rrr} 1 & -1 & 1 \\ 0 & 1 & -2 \\ 0 & 0 & 1 \end{array}\right]
  \cdot  \left[ \begin{array}{r} w_{g,1}+ w_{g,2}+w_{g,3} \\ 
   a_2^k (w_{g,2}+2 w_{g,3}) \\ a_3^k \; w_{g,3} \end{array} \right]_{\pml{p_0}} \\
   & = & \left[ \begin{array}{r} (w_{g,1}+ w_{g,2}+w_{g,3}) - a_2^k(w_{g,2}+2w_{g,3}) + a_3^k \; w_{g,3} \\ 
   a_2^k (w_{g,2}+2 w_{g,3}) - 2 a_3^k \; w_{g,3} \\ a_3^k \; w_{g,3} \end{array} \right]_{\pml{p_0}}
\end{eqnarray*}
 
 Right away we observe that the asymptotic ratio $w_{g,1}(\infty)$ of the gap $g$ to the gap $2$ is the sum
 of the initial ratios of all of $g$'s driving terms.  We also observe that the ratio converges
 to the asymptotic value as quickly as $a_2^k \fto 0$.  While $a_3^k$ becomes small pretty
 quickly, the convergence of $a_2^k$ is slow. Figure~\ref{AjkFig} plots $a_2^k$ and $a_3^k$
 for $p_0=13$ up to $p \approx 3\cdot 10^{15}$.

\begin{figure}[t]
\centering
\includegraphics[width=5in]{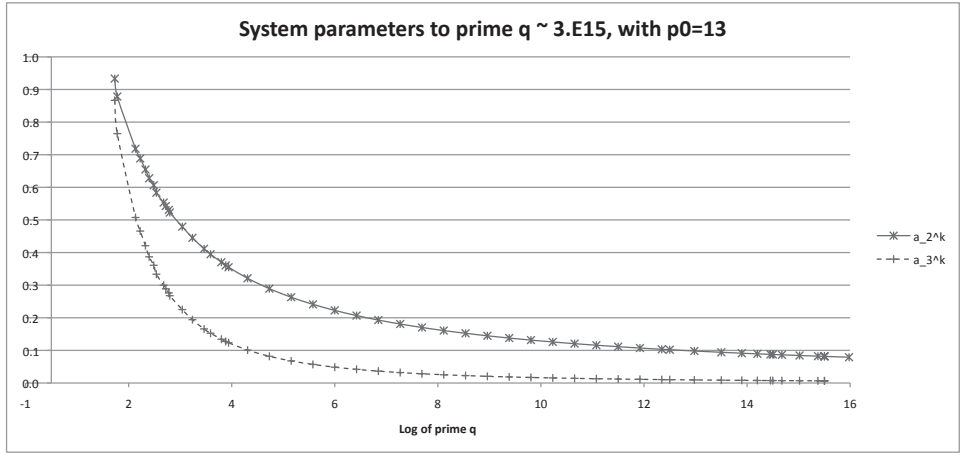}
\caption{\label{AjkFig} A graph of $a_2^k$ and $a_3^k$, with $p_0=13$, up to 
about $p \approx 3\cdot 10^{15}$. The dominant eigenvalue for $M_J$ is $1$, the second
eigenvalue is $a_2$ and the third $a_3$.  So the rate of convergence to the asymptotic 
ratio $w_{g,1}(\infty) = N_g / N_2$ is governed by how quickly $a_2^k \fto 0$. }
\end{figure}
 
 Let us apply the above analysis to all of the gaps that satisfy the required conditions.  Using
 $p_0=5$, we observe in $\pgap(\pml{5})=64242462$ that the smallest sum of a constellation
 of length $4$ is $12$, for $s=2424$.  So we cannot apply the analysis to $g=12$, but we can
 for the gaps $g = 6, \; 8, \; 10$.
 
 \begin{center}
 \begin{tabular}{|r|rrr|r|} \hline
\multicolumn{4}{|c|}{$\pgap(\pml{5})=64242462$ } & \\
$g$ & $w_{g,1}(\pml{5})$ & $w_{g,2}(\pml{5})$ & $w_{g,3}(\pml{5})$ & $w_{g,1}(\infty)$ \\  \hline
$6$ & $2/3$ & $4/3$ & $0$ & $2$ \\
$8$ & $0$ & $2/3$ & $1/3$ & $1$ \\
$10$ & $0$ & $2/3$ & $2/3$ & $4/3$ \\ \hline
 \end{tabular}
\end{center}

To obtain the asymptotic ratio $w_{g,1}(\infty)$, we simply add together the initial 
ratios of all driving terms.  These results tell us that as quickly as $a_2^k \fto 0$,
the number of occurrences of the gap $g=6$ in the
cycle of gaps $\pgap(\pml{p})$ for Eratosthenes sieve approaches double the number
of gaps $g=2$.  Despite having driving terms of length two and of length three, the number of
gaps $g=8$ approaches the number of gaps $g=2$ in later stages of the sieve, and the
ratio of the number of gaps $g=10$ to the number of gaps $g=2$ approaches $4/3$.

These are {\em not} probabilistic estimates.  These ratios are based on the actual counts of the 
populations of these gaps and their driving terms across stages of Eratosthenes sieve.

\section{First set of results on gaps and constellations for Eratosthenes sieve}
Before developing the general dynamic system that describes the population of any gap $g=2k$
in the cycle of gaps, we pause to list the following results which we can already establish.

\begin{itemize}
\item {\em Twin Generators}. The number of gaps $g=2$ in the cycle of gaps
$\pgap(\pml{p_k})$ is
$$N_2(\pml{p_k}) = \prod_{q=3}^{p_k} (q-2).$$  
{\bf Proof:} This is a direct application of Corollary~\ref{CorGrowth}.
\item {\em Primorial conjecture for $6=\pml{3}$}.  From our work above on the $3$-dimensional
dynamic system, we have calculated the asymptotic ratio between the numbers of gaps $6=\pml{3}$
and $2=\pml{2}$,
$$ w_{6,1}(\infty) = \lim_{q \fto \infty} \frac{N_{\pml{3}}(\pml{q})}{N_{\pml{2}}(\pml{q})} = 2$$
But $2 = \frac{3-1}{3-2}$, which is the ratio of occurrences implied by Hardy and Littlewood's 
Conjecture B.
\item {\em ET Spikes}.  For gaps in the cycles of gaps,
$$ \lim \sup \frac{g_{i+1}}{g_i} = \infty \gap {\rm and} \gap \lim \inf \frac{g_{i+1}}{g_i} = 0.$$
In particular, for $g_i=2$, the adjacent gaps $g_{i-1}$ and $g_{i+1}$ become arbitrarily large in 
later stages of the sieve. 

{\bf Proof:} This is a direct result of Theorem~\ref{DelThm}.  
Let $s_k = 2\tilde{g}_k$
be the constellation of length two in $\pgap(\pml{p_k})$ with first gap $2$ and second gap 
$\tilde{g}_k$ such that $\tilde{g}_k$ is the largest that occurs among all such constellations of length two.
$$\tilde{g}_k = \max \left\{ g \st s=2g \; {\rm occurs \; in} \; \pgap(\pml{p_k}) \right\}$$  
In forming
$\pgap(\pml{p_{k+1}})$ the closure to the right of $\tilde{g}_k$ occurs exactly once.
Thus $\tilde{g}_{k+1} > \tilde{g}_k$, and we have our result for $\limsup$. 

By symmetry of $\pgap(\pml{p})$
the constellation $\tilde{g}_k 2$ also occurs, and we thereby have our result for $\liminf$.
\item {\em ET Superlinear growth}.  For every $k > 1$ there is a cycle of gaps $\pgap(\pml{p})$
with a constellation of $k$ consecutive gaps such that
$$ g_{n+1} < g_{n+2} < \cdots < g_{n+k}.$$
This constellation persists across all subsequent stages of the sieve, and its population increases
by the factor $p-k-1$ at each stage.  

{\bf Proof:}   In the middle of the cycle of gaps $\pgap(\pml{p_i})$ there occurs the
constellation
$$ \tilde{s}_j = 2^j \; 2^{j-1} \ldots 842 \; 4 \; 248 \ldots 2^{j-1} \; 2^j$$
in which $j$ is the smallest number such that $2^{j+1} > p_{i+1}$.  
For a given $k$, we take $p_i$ large enough so that $j > k$.  Then the right half of $\tilde{s}_j$
is the desired constellation, and it satisfies the condition $|s| = 2^{j+1}-2 < 2 p_{i+1}.$

The persistence of the constellation and the growth of its population by the factor
$p-k-1$ is given by Corollary~\ref{CorGrowth}. 
\item {\em ET Superlinear decay}.  For every $k > 1$ there is a cycle of gaps $\pgap(\pml{p})$
with a constellation of $k$ consecutive gaps such that
$$g_{n+1} >  g_{n+2} > \cdots  > g_{n+k}.$$
This constellation persists across all subsequent stages of the sieve, and its population increases
by the factor $p-k-1$ at each stage.  {\bf Proof:}  The desired constellation is the left 
half of $\tilde{s}_j$.
\end{itemize}

\section{A model for populations across iterations of the sieve}
We now identify a discrete dynamic system that provides exact counts
of a gap and its driving terms.  These raw counts grow
superexponentially, and so to better understand their behavior we take the ratio
of a raw count to the number of gaps $g=2$ at each stage of the sieve.
In the work above we created and examined this dynamic system for driving
terms up to length $3$.  Here we generalize this approach by considering driving
terms up to length $J$, for any $J$.

Fix a sufficiently large size $J$.
For any gap $g$ that has driving terms of lengths up to $j$, with $j \le J$, we form a vector
of initial values $\left. \bar{w}\right|_{p_{0}}$, whose $i^{\rm th}$ entry is the ratio
of the number of driving terms for $g$ of length $i$ in $\pgap(\pml{p_0})$
to the number of gaps $2$ in this cycle of gaps.

Generalizing our work for $J=3$ above, we model the population of the gap $g$ and its
driving terms across stages of Eratosthenes sieve as a discrete dynamic system.
\begin{eqnarray*}
\left. \bar{w}\right|_{\pml{p_k}} & = & M_J(p_k) \cdot \left. \bar{w}\right|_{\pml{p_{k-1}}} \\
 & = & \left[\begin{array}{cccccc}
1 & b_1 & 0 & \cdots & & 0 \\
0  & a_2 & b_2 & \ddots &  & 0 \\
 & 0 & a_3 & b_3 & \ddots & 0 \\
  \vdots & & \ddots & \ddots & \ddots & 0 \\
  0 & & \cdots & & a_{J-1} & b_{J-1} \\
  0 & & \cdots & & 0 & a_J \end{array} \right]_{p_k} \cdot \left. \bar{w}\right|_{\pml{p_{k-1}}} 
\end{eqnarray*}
in which
\begin{equation}\label{Eqaj}
a_j(p)  =  \frac{p-j-1}{p-2} \biggap {\rm and} \biggap b_j(p) = \frac{j}{p-2}.
\end{equation}

Iterating this discrete dynamic system from the initial conditions at $p_0$ up through $p_k$, we have
\begin{eqnarray*}
\left. \bar{w}\right|_{\pml{p_{k}}} & = & \left. M_J \right|_{p_{k}} \cdot \left. \bar{w}\right|_{\pml{p_{k-1}}} \\
 & = & M_J^k  \cdot \left. \bar{w}\right|_{\pml{p_{0}}} 
\end{eqnarray*}
The matrix $M_J$ does not depend on the gap $g$.  It does depend on the prime $p_k$,
and we use the exponential notation $M_J^k$ to indicate the product of the $M$'s over
the indicated range of primes.

That $M_J$ does not depend on the gap $g$ is interesting.  This means that the recursion 
treats all gaps fairly.  The recursion itself is not biased toward certain gaps or constellations.
Once a gap has driving terms in $\pgap(\pml{p})$, the populations across all further stages of 
the sieve are completely determined.

$M_J^k$ applies to all constellations whose driving terms have length $j \le J$; and we 
continue to use the exponential notation to denote the product over the sequence of primes from
$p_1$ to $p_k$: e.g.
$$ M_J^k =  \MJP{J}{p_k} \cdot  \MJP{J}{p_{k-1}} \cdots \MJP{J}{p_1}.
$$ 
With $M_J^k$ we can calculate the ratios $\Rat{j}{g}(p_k)$
for the complete system of driving terms, relative to the
population of the gap $2$, for the cycle of gaps $\pgap(\pml{p_k})$ (here, $p_k$ is the $k^{\rm th}$
prime after $p_0$).  With $J=3$ we calculated above the ratios for $g=6, \; 8, \; 10.$
For $g=12$ we need $J=4$, and for $g=30$, we need $J=8$.

Fortunately, we can completely describe the eigenstructure for $\left. M_J\right|_{p}$,
and even better -- {\em the eigenvectors for $M_J$ do not depend on the prime $p$}.  This means
that we can use the eigenstructure to provide a simple description of the behavior of this 
iterative system as $k \fto \infty$.

\subsection{Eigenstructure of $M_J$}
We list the eigenvalues, the left eigenvectors and the right eigenvectors for $M_J$, writing these
in the product form
$$ M_J  =  R \cdot \Lambda \cdot L $$
with $LR = I$.
For the general system $M_J$, the upper triangular entries of $R$ and $L$ are binomial coefficients,
with those in $R$ of alternating sign; and the eigenvalues are the $a_j$ defined in Equation~\ref{Eqaj}
above.
\begin{eqnarray*}
R_{ij} & = & \left\{ \begin{array}{cl}
(-1)^{i+j}\Bi{j-1}{i-1} & {\rm if} \; i \le j \\ & \\
0 & {\rm if} \; i > j \end{array} \right. \\
& & \\
\Lambda & = & {\rm diag}(1, a_2, \ldots, a_J) \\
 & & \\
L_{ij}  & = & \left\{ \begin{array}{ll}
\Bi{j-1}{i-1} & {\rm if} \; i \le j \\ & \\
0 & {\rm if} \; i > j \end{array} \right. 
\end{eqnarray*}

For any vector $\bar{w}$, multiplication by the left eigenvectors 
(the rows of $L$) yields the 
coefficients for expressing this vector of initial conditions over the basis given by the 
right eigenvectors (the columns of $R$):
$$ \bar{w} =  (L_{1 \cdot} \bar{w}) R_{\cdot 1} + \cdots + (L_{J \cdot} \bar{w}) R_{\cdot J}
$$

\begin{lemma}\label{AsymLemma}
Let $g$ be a gap and $p_0$ a prime such that $g < 2 p_1$.
In $\pgap(\pml{p_0})$ let the initial ratios for $g$ and its driving terms be given by
$\wgp{g}{\pml{p_0}}$.
Then the ratio of occurrences of this gap $g$
to occurrences of the gap $2$ in $\pgap(\pml{p})$ as $p \fto \infty$ converges to
the sum of the initial ratios across the gap $g$ and all its driving terms:
$$
\Rat{1}{g}(\infty) = L_{1 \cdot} \wgp{g}{\pml{p_0}} = \sum_j \left. \Rat{j}{g} \right|_{\pml{p_0}}.
$$
\end{lemma}

\begin{proof}
Let $g$ have driving terms up to length $J$.  Then the ratios $\wgp{g}{\pml{p}}$ 
are given by  the iterative linear system
$$
\wgp{g}{\pml{p_k}} = M_J^k \cdot  \wgp{g}{\pml{p_0}}.
$$
From the eigenstructure of $M_J$, we have
$$
\wgp{g}{\pml{p_0}} = (L_1 \wgp{g}{\pml{p_0}})R_1 + (L_2 \wgp{g}{\pml{p_0}})R_2 
+ \cdots + (L_J \wgp{g}{\pml{p_0}})R_J ,
$$
and so
\begin{eqnarray} \label{EqDynSys}
M_J^k \wgp{g}{\pml{p_0}} &=& 
  (L_1 \wgp{g}{\pml{p_0}})R_1 + a_2^k (L_2 \wgp{g}{\pml{p_0}})R_2  + \cdots \\ 
 & &  \cdots + a_J^k (L_J \wgp{g}{\pml{p_0}})R_J.  \nonumber
\end{eqnarray}
We note that $L_{1 \cdot} = [1 \cdots 1]$, $\lambda_1 = 1$, and $R_{\cdot 1} = e_1$;
that the other eigenvalues $a_j^k \fto 0$ with $a_j^k > a_{j+1}^k$.   Thus as $k \fto \infty$
the terms on the righthand side decay to $0$ except for the first term, establishing the result.
\end{proof}

With Lemma~\ref{AsymLemma} and the initial values in $\pgap(\pml{13})$ 
in Table~\ref{G13Table}, we can calculate the asymptotic ratios
of the occurrences of the gaps $g = 6, 8, \ldots, 32$ to the gap $g=2$. 

\renewcommand\arraystretch{1.2}
\begin{table}
\begin{center}
\begin{tabular}{ll} \hline
\multicolumn{2}{|c|}{Values of $a_j^k$ at $p_k=999,999,999,989$ for $p_0 = 13$} \\ \hline
 &
$a_2^k = 0.10206751799779$ \\
 & $a_3^k = 0.01019996897567$ \\
$ a_j^k = \prod_{q=17}^{p_k} \frac{q-j-1}{q-2} $ \biggap & $a_4^k = 0.00099592269918$ \\
 & $a_5^k = 0.00009477093531$ \\
 & $a_6^k = 0.00000876214163 $ \\
 & $a_7^k = 0.00000078408120$ \\
 & $a_8^k =  0.00000006757562$ \\
 & $a_9^k =  0.00000000557284$ \\ \hline
\end{tabular}
\caption{\label{AjkTable} Calculated values of the eigenvalues $a_j^k$ up to $p_k \approx 10^{12}$.  
If we use initial conditions
from $\pgap(\pml{13})$, then $p_0 = 13$ and the products start with $p_1=17$.  }
\end{center}
\end{table}
\renewcommand\arraystretch{1}

From the calculated values in Table~\ref{AjkTable}, 
we see the decay of the $a_j^k$ toward $0$, but $a_2^k$ and $a_3^k$ 
are still making significant contributions when $p_k \approx 10^{12}$.

In Table~\ref{G13Table} we used $p_0=13$ for our initial conditions since the prime 
$p=13$ is the first prime 
for which the conditions of Corollary~\ref{CorGrowth} 
are satisifed for the next primorial $g=\pml{5} = 30$.  

\renewcommand\arraystretch{0.8}
\begin{table}
\begin{center}
\begin{tabular}{|r|rrrrrrrrr|} \hline
gap & \multicolumn{9}{c|}{ $n_{g,j}(13)$: driving terms of length $j$ in $\pgap(\pml{13})$ }\\ [1 ex]
$g$ & $j=1$ & $2$ & $3$ & $4$ & $5$ & $6$ & $7$ & $8$ & $9$ \\ \hline 
 $\lil 2, \; 4$ & $\lil 1485$ & & & & & & & & \\
 $\lil 6$ & $\lil 1690$ & $\lil 1280$ & & & & & & & \\
 $\lil 8$ & $\lil 394$ & $\lil 902$ & $\lil 189$ & & & & & & \\
 $\lil 10$ & $\lil 438$ & $\lil 1164$ & $\lil 378$ & & & & & & \\
 $\lil 12$ & $\lil 188$ & $\lil 1276$ & $\lil 1314$ & $\lil 192$ & & & & & \\
 $\lil 14$ & $\lil 58$ & $\lil 536$ & $\lil 900$ & $\lil 288$ & & & & & \\
 $\lil 16$ & $\lil 12$ & $\lil 252$ & $\lil 750$ & $\lil 436$ & $\lil 35$ & & & &  \\
 $\lil 18$ & $\lil 8$ & $\lil 256$ & $\lil 1224$ & $\lil 1272$ & $\lil 210$ & & & & \\
 $\lil 20$ & $\lil 0$ & $\lil 24$ & $\lil 348$ & $\lil 960$ & $\lil 600$ & $\lil 48$ & & & \\
 $\lil 22$ & $\lil 2$ & $\lil 48$ & $\lil 312$ & $\lil 784$ & $\lil 504$ & & & & \\
 $\lil 24$ & $\lil 0$ & $\lil 20$ & $\lil 258$ & $\lil 928$ & $\lil 1260$ & $\lil 504$ & & & \\
 $\lil 26$ & $\lil 0$ & $\lil 2$ & $\lil 40$ & $\lil 322$ & $\lil 724$ & $\lil 448$ & $\lil 84$ & & \\
 $\lil 28$ & $\lil 0$ & $\lil 0$ & $\lil 36$ & $\lil 344$ & $\lil 794$ & $\lil 528$ & $\lil 80$ & & \\
 $\lil 30$ & $\lil 0$ & $\lil 0$ & $\lil 10$ & $\lil 194$ & $\lil 1066$ & $\lil 1784$ & $\lil 816$ & $\lil 90$ & \\
 $\lil 32$ & $\lil 0$ & $\lil 0$ & $\lil 0$ & $\lil 12$ & $\lil 200$ & $\lil 558$ & $\lil 523$ & $\lil 172$ & $\lil 20$ \\ \hline
 \end{tabular}
\caption{ \label{G13Table} For small gaps $g$, 
this table lists the number of gaps and driving terms of length $j$
 that occur in the cycle of gaps $\pgap(\pml{13})$. We can use these as initial conditions
 for the population model in Equation~\ref{EqDynSys} of size $J \le 9$.}
 \end{center}
 \end{table}
\renewcommand\arraystretch{1}

Applying Lemma~\ref{AsymLemma} to the initial conditions for $p_0=13$ in Table~\ref{G13Table}, 
as $p_k \fto \infty$, the following ratios
describe the relative frequency of occurrence of these gaps in Eratosthenes sieve:
\begin{center}
\begin{tabular}{rl}\hline
ratio $w_{g,1}(\infty)$ : & gaps $g$ with this ratio \\ \hline
$1$ : & $2, \; 4, \; 8, \; 16, \; 32$  \\
$1.\overline{09}$ : & $26$ \\
$1.\bar{1}$ : & $22$ \\
$1.2$ : & $14, \; 28$ \\
$1.\bar{3}$ : & $10, \; 20$ \\
$2$ : & $ 6, \; 12, \; 18, \; 24$ \\
$ 2.\bar{6}$ : & $ 30$
\end{tabular}
\end{center}

This table begins to suggest that the ratios implied by Hardy and Littlewood's Conjecture B may
hold true in Eratosthenes sieve.

The ratios discussed in this paper give the exact values of the relative frequencies of various gaps and
constellations as compared to the number of gaps $2$ at each stage of Eratosthenes sieve.
For gaps between primes, if the closures are at all fair as the sieving process continues, then 
these ratios in stages of the sieve should
also be good indicators of the relative occurrence of these gaps and constellations  
among primes.

\subsection{Primorial $g=\pml{5}=30$.}
We also see that the primorial $g=\pml{5}=30$ eventually becomes more numerous than 
$g=\pml{3}=6$.  However, if we apply the expansion of the higher order terms
\begin{eqnarray}\label{EqEigSys}
\wgp{g}{\pml{p_k}} & = & M_J^k \wgp{g}{\pml{p_0}} \nonumber \\
& = &  (L_1 \wgp{g}{\pml{p_0}})R_1 + a_2^k (L_2 \wgp{g}{\pml{p_0}})R_2  + \cdots 
  + a_J^k (L_J \wgp{g}{\pml{p_0}})R_J 
 \end{eqnarray}
 to the initial conditions for $p_0=13$, using the $a_j^k$ as tabulated for 
 $\hat{p}=999,999,999,989$, we calculate that for $p_k \approx 10^{12}$
 $$ \wgp{6}{\pml{\hat{p}}} \approx 1.912 \biggap {\rm and} \biggap
\wgp{30}{\pml{\hat{p}}} \approx 1.579.$$
The convergence to the asymptotic values is very slow, due primarily to the slow decay of 
$a_2^k$.

For what $p$ will $\wgp{30}{\pml{p}} > \wgp{6}{\pml{p}}$?  That is, when will the gap $30$ be
more numerous in Eratosthenes sieve than the gap $6$? 
To estimate this, we examine the first coordinate across the terms in Equation~\ref{EqEigSys}.
The first coordinates of the $R$'s are $1$'s of alternating sign, and so we have
$$ \wgp{g}{\pml{p}} = (L_1 \wgp{g}{\pml{p_0}}) - a_2^k (L_2 \wgp{g}{\pml{p_0}})  + \cdots 
  + (-1)^{J+1} a_J^k (L_J \wgp{g}{\pml{p_0}}).$$
To estimate where $\wgp{30}{\pml{p}}-\wgp{6}{\pml{p}}>0$, we make the rough approximation
that for $p_k >> j$,
$$ a_j^k \approx (a_2^k)^{j-1}$$
and solve for the parameter $a_2^k$.  Using the data from $p_0 = 13$, we see that
$\wgp{30}{\pml{p}} > \wgp{6}{\pml{p}}$, that is, that the gaps $30$ will finally be more numerous
in Eratosthenes sieve than the gaps $6$, when $a_2^k < 0.06275$.

For $p_0=13$, when $p_k \approx 10^{12}$ the parameter $a_2^k \approx 0.1$, and when
$p_k \approx 10^{15}$ the parameter $a_2^k \approx 0.08$.  The decay of $a_2^k$ is so slow 
that there will still be fewer gaps $30$ than gaps $6$ in Eratosthenes sieve when $p \approx 10^{15}$.

\section{Polignac's conjecture and\\  Hardy \& Littlewood's Conjecture B}
At this point, here is what we know about the population of a gap $g$ through the cycles of gaps
$\pgap(\pml{p})$ in Eratosthenes sieve.  We need $p_0$ such that conditions of Corollary~\ref{CorGrowth}
hold, in particular the condition $g < 2 p_1$.  
Then we need $J$ such that no constellation of length $J+1$ has sum
equal to $g$.  This is the size of system we need to consider, to apply the dynamic system 
of Equation~(\ref{EqDynSys}) to $g$ and its driving terms.  For a given $g$, once we have $p_0$ and $J$, 
from $\pgap(\pml{p_0})$ we can obtain counts of driving terms for $g$
from length $1$ to $J$ to create the initial conditions $w_g(\pml{p_0})$, and we can
apply the model directly or through its eigenstructure, to obtain the exact populations 
of $g$ and its driving terms through the all further
stages of Eratosthenes sieve.

Our progress along this line of increasing $p_0$ and $J$ is complicated primarily by our
having to construct $\pgap(\pml{p_0})$.  This cycle of gaps contains $\phi(\pml{p_0})$ elements,
which grows unmanageably large.  If we have $\pgap(\pml{p_0})$, then for every gap $g < 2 p_1$ 
we can enumerate the driving terms of various 
lengths.  We take the maximum such length as $J$.

In this section we introduce an alternate way to obtain initial conditions for any gap $g$,
sufficient to apply Lemma~\ref{AsymLemma}.

As an analogue to Polignac's conjecture, we show that for any even number $2n$, the
gap $g=2n$ or its driving terms occur at some stage of Eratosthenes sieve, and we show
that although we can't apply the complete dynamic system, we do have enough
information to get the asymptotic result from Lemma~\ref{AsymLemma}.

{\em Polignac's Conjecture}:  For any even number $2n$, there are infinitely many prime pairs
$p_j$ and $p_{j+1}$ such the difference $p_{j+1}-p_j = 2n$.

In Theorem~\ref{PolThm} below we establish an analogue of Polignac's conjecture for 
Eratosthenes sieve, that for any number $2n$ the gap $g=2n$ occurs infinitely often
in Eratosthenes sieve, and the ratio of occurrences of this gap to the gap $2$ approaches
the ratio implied by Hardy \& Littlewood's Conjecture B:
$$ w_{g,1}(\infty) = \lim_{p \rightarrow \infty} \frac{N_g(\pml{p})}{N_2(\pml{p})}
= \prod_{q>2, \; q|g} \frac{q-1}{q-2}.$$

To obtain this result, we first consider $\Z \bmod Q$ and its cycle of gaps $\pgap(Q)$,
in which $Q$ is the product of the prime divisors of $2n$.
We then bring this back into Eratosthenes sieve by filling in the primes
missing from $Q$ to obtain a primorial $\pml{p}$.

Once we are working with $\pgap(\pml{p})$, the condition $g < 2 p_{k+1}$
may still prevent us from applying Corollary~\ref{CorGrowth}.  
However, we are able to show that we have enough information to 
apply Lemma~\ref{AsymLemma} under the construction we are using.

\subsection{General recursion on cycles of gaps}
We need to develop a more general form of the recursion on cycles of gaps, one
that applies to creating $\pgap(qN)$ from $\pgap(N)$ for any prime $q$ and
number $N$.  We also need a variant of Lemma~\ref{Lemma2p} that does
not require the condition $g < 2 p_{k+1}$.

Let $\pgap(N)$ denote the cycle of gaps among the generators in $\Z \bmod N$, with
the first gap being that between $1$ and the next generator.  There are
$\phi(N)$ gaps in $\pgap(N)$ that sum to $N$. In our work in the preceding sections, 
we focused on Eratosthenes sieve, in which $N=\pml{p}$, the primorials. 

There is a one-to-one correspondence between generators of $\Z \bmod N$ and the gaps in 
$\pgap(N)$.  Let
$$\pgap(N) = g_1 \; g_2 \; \ldots g_{\phi(N)}.$$
Then for $k < \phi(N)$, $g_k$ corresponds to the generator $\gamma = 1+\sum_{j=1}^{k} g_j$, and since
$\sum_{j=1}^{\phi(N)} = N$, the generator $1$ corresponds to $g_{\phi(N)}$.  Moreover, since
$1$ and $N-1$ are always generators, $g_{\phi(N)}=2$.  For any generator $\gamma$, $N-\gamma$ 
is also a generator, which implies that except for the final $2$, $\pgap(N)$ is symmetric.
As a convention, we write the cycles with the first gap being from $1$ to the next generator. 

We build $\pgap(N)$ for any $N$ by introducing 
one prime factor at a time.

\begin{lemma}\label{LemmaCycle}
Given $\pgap(N)$, for a prime $q$ we construct $\pgap(qN)$ as follows:
\begin{enumerate}
\item[a)] if $q \mid N$, then we concatenate $q$ copies of $N$,
$$ \pgap(qN) = \underbrace{\pgap(N) \cdots \pgap(N)}_{q \gap {\rm copies}}$$
\item[b)] if $q \not| N$, then we build $\pgap(qN)$ in three steps:
\begin{enumerate}
\item[R1] Concatenate $q$ copies of $\pgap(N)$;
\item[R2] Close at $q$;
\item[R3] Close as indicated by the element-wise product $q * \pgap(N)$.
\end{enumerate}
\end{enumerate}
\end{lemma}

\begin{proof}
A number $\gamma$ in $\Z \bmod N$ is a generator iff $\gcd(\gamma,N)=1$. 
\begin{itemize}
\item[a)] 
Assume $q|N$.  Since $\gcd(\gamma,N)=1$, we know that $q \not| \gamma$.

For $j=0,1,\ldots,q-1$, we have 
$$\gcd(\gamma+jN, qN)= \gcd(\gamma,qN) = \gcd(\gamma,N)=1.$$
Thus $\gcd(\gamma,N)=1$ iff $\gcd(\gamma+jN,qN)=1$, and so the generators of $\Z \bmod qN$ have
the form $\gamma+jN$, and the gaps take the indicated form.
\item[b)] 
If $q \not| N$ then we first create a set of candidate generators for $\Z \bmod qN$, by
considering the set 
$$\set{\gamma+jN \st \gcd(\gamma,N)=1, \gap j=0,\ldots, q-1}.$$
For gaps, this is the equivalent of step R1, concatenating $q$ copies of $\pgap(N)$.
The only prime divisor we have not accounted for is $q$; if $\gcd(\gamma+jN,q)=1$, then this
candidate $\gamma+jN$ is a generator of $\Z \bmod qN$.  So we have to remove $q$ and its
multiples from among the candidates.

We first close the gaps at $q$ itself.  We index the gaps in the $q$ concatenated copies of
$\pgap(N)$:
$$ g_1 g_2 \ldots g_{\phi(N)} \ldots g_{q\cdot \phi(N)}.$$
Recalling that the first gap $g_1$ is the gap between the generator $1$ and the next smallest
generator in $\Z \bmod N$, the candidate generators are the 
running totals $\gamma_j = 1+\sum_{i=1}^{j-1} g_i$.  We take the $j$ for which $\gamma_j = q$, 
and removing $q$ from the list of candidate generators corresponds to replacing the gaps
$g_{j-1}$ and $g_j$ with the sum $g_{j-1}+g_j$.  This completes step R2 in the construction.

To remove the remaining multiples of $q$ from among the candidate generators, we note
that any multiples of $q$ that share a prime factor with $N$ have already been removed.
We need only consider multiples of $q$ that are relatively prime to $N$; that is, we only need
to remove $q\gamma_j$ for each generator $\gamma_j$ of $\Z \bmod N$ by closing the
corresponding gaps.

We can perform these closures by working directly with the cycle of gaps $\pgap(N)$.  Since
$q\gamma_{i+1}-q\gamma_{i} = qg_{i}$, we can go from one closure to the next by tallying the
running sum from the current closure until that running sum equals $qg_{i}$.
Technically, we create a series of indices beginning with $i_0=j$ such that $\gamma_j=q$,
and thereafter $i_k=j$ for which $\gamma_j-\gamma_{i_{k-1}}=q\cdot g_k$.  To cover the cycle
of gaps under construction, which consists initially of $q$ copies of $\pgap(N)$, $k$ runs
only from $0$ to $\phi(N)$.  We note that the last interval wraps around the end of the cycle
and back to $i_0$:  $i_{\phi(N)}=i_0$.
\end{itemize}
\end{proof}

\begin{theorem}\label{ThmqN}
In step R3 of Lemma~\ref{LemmaCycle}, each possible closure in $\pgap(N)$ occurs exactly once
in constructing $\pgap(qN)$.
\end{theorem}

\begin{proof}
Consider each gap $g$ in $\pgap(N)$.  Since $q \not| N$, $N \bmod q \neq 0$.
Under step R1 of the construction, $g$ has $q$ images.  Let the generator corresponding
to $g$ be $\gamma$.  Then the generators corresponding to the images of $g$ under step
R1 is the set:
$$\left\{ \gamma+jN \st j=0,\ldots,q-1\right\}.$$
Since $N \bmod q \neq 0$, there is exactly one $j$ for which $(\gamma+jN) \bmod q = 0$.
For this gap $g$, a closure in R2 and R3 occurs once and only once, at the image corresponding
to the indicated value of $j$.
\end{proof}

\begin{figure}[t]
\centering
\includegraphics[width=5in]{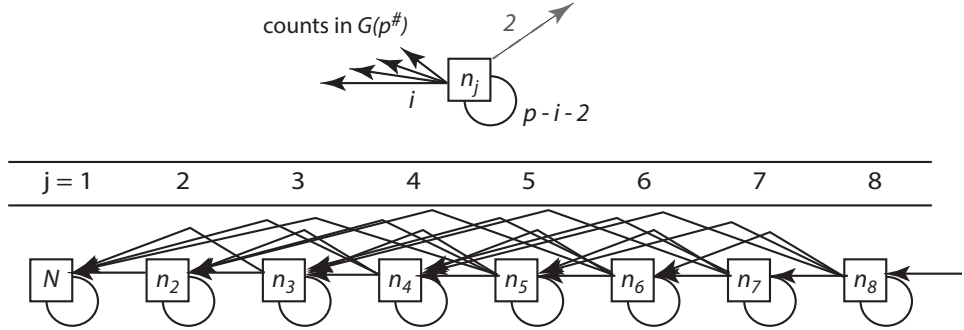}
\caption{\label{GenSystemFig} In the general dynamic system, when the 
condition $g < 2 p_{k+1}$ may not be satisfied, the interior closures may not
occur in distinct copies of the constellation.  However, the two exterior closures
still remove two copies from being driving terms for $g$.  The other $n_j-2$ copies
remain as driving terms, but we cannot specify their lengths. }
\end{figure}

\begin{corollary}\label{qNCor}
Let $g$ be a gap.  If for the prime $q$, $q \not| g$, then 
$$ \sum \Rat{j}{g}(qN) = \sum \Rat{j}{g}(N).$$
\end{corollary}

\begin{proof}
Consider a driving term $s$ for $g$, of length $j$ in $\pgap(N)$.  
In constructing $\pgap(qN)$, we initially create $q$ copies of $s$.

If $q | N$, then the construction is complete.  For each driving term for $g$
in $\pgap(N)$ we have $q$ copies, and so
$ n_{g,j}(qN) = q  \cdot n_{g,j}(N).$  However, we also have $q$ copies of
every gap $2$ in $\pgap(N)$, 
 $n_{2,1}(qN) = q \cdot n_{2,1}(N)$.
Thus $\Rat{j}{g}(qN) = \Rat{j}{g}(N)$, and we have equality for each length $j$, and 
so the result about the sum is immediate.

If $q \not| N$, then in step R1 we create $q$ copies of $s$.
In steps R2 and R3, each of the possible closures in $s$ occurs once, distributed among the
$q$ copies of $s$.  The $j-1$ closures interior to $s$ change the lengths of some of the driving
terms but don't change the sum, and the result
is still a driving term for $g$.  Only the two exterior closures, one at each end of $s$, change
the sum and thereby remove the copy from being a driving term for $g$.  Since $q \not| g$, 
these two exterior closures occur in separate copies of $s$.  See Figure~\ref{GenSystemFig}.

If the condition $g < 2p_{k+1}$ applies, then each of the closures occur in a separate copy of 
$s$, and we can use the full dynamic system of Corollary~\ref{CorGrowth}.
For the current result we do not know that the closures necessarily occur in distinct copies of 
$s$, and so we can't be certain of the lengths of the resulting constellations.

However, we do know that of the $q$ copies of $s$, two are eliminated as driving terms and
$q-2$ remain as driving terms of various lengths.  
$$\sum_j n_{g,j}(qN) = (q-2) \sum_j n_{g,j}(N).$$
Since $n_{2,1}(qN) = (q-2)n_{2,1}(N)$, the ratios are preserved
$$\sum_j \Rat{j}{g}(qN) = \sum_j \Rat{j}{g}(N).$$
\end{proof}

By combining the preceding Corollary~\ref{qNCor} with Lemma~\ref{AsymLemma},
we immediately obtain the following result, that for any gap $g$, if we look at its largest prime
factor $\bar{q}$, then we can calculate the asymptotic ratios from $\pgap(\pml{\bar{q}})$.

\begin{corollary}\label{InfCor}
Let $g=2n$ be a gap, and let $\bar{q}$ be the largest prime factor of $g$.  Then
$$ \Rat{1}{g}(\infty) = \sum \Rat{j}{g}(\pml{\bar{q}}).$$
\end{corollary}

\begin{proof}
For all primes $p > \bar{q}$, by Corollary~\ref{qNCor}
$$ \sum \Rat{j}{g}(\pml{p}) = \sum \Rat{j}{g}(\pml{\bar{q}}),$$
so once we reach $\pgap(\pml{\bar{q}})$, we continue through additional stages of the sieve
if necessary until the condition $g < 2p_1$ is satisfied, but the ratios remain unchanged during
this formality.  So the result from Lemma~\ref{AsymLemma} can be obtained from the
ratios determined in $\pgap(\pml{\bar{q}})$.
\end{proof}

\subsection{Polignac's conjecture for Eratosthenes sieve}
We establish an equivalent of Polignac's conjecture for Eratosthenes sieve.

\begin{theorem}\label{PolThm}
For every $n>0$, the gap $g=2n$ occurs infinitely often in Eratosthenes sieve, and the
ratio of the number of occurrences of $g=2n$ to the number of $2$'s converges asymptotically
to 
$$ \Rat{1}{2n}(\infty) = \prod_{q>2, \; q|n} \frac{q-1}{q-2}.
$$
\end{theorem}

We establish this result in two steps.  
First we find a stage of Eratosthenes sieve in which
the gap $g=2n$ has driving terms.  
Once we can enumerate the driving terms for $g$ 
in this initial stage of Eratosthenes sieve,
we can establish the asymptotic ratio of gaps $g=2n$ to the gaps $g=2$ as the sieve continues.

\begin{lemma}\label{QLem}
Let $g=2n$ be given.  Let $Q$ be the product of the primes dividing $2n$, including $2$.
$$Q = \prod_{q | 2n} q.$$
Finally, let $\bar{q}$ be the largest prime factor in $Q$.

Then in $\pgap(\pml{\bar{q}})$ the gap $g$ has driving terms, the total number 
of which satisfies
$$ \sum_j n_{g,j}(\pml{\bar{q}}) = \phi(Q) \cdot \prod_{p < \bar{q}, \; p\; \nmid \; Q} (p-2).$$
\end{lemma}

\begin{proof}
Let $n_1 = 2n / Q$.  By Lemma~\ref{LemmaCycle} the cycle of gaps $\pgap(2n)$ consists of $n_1$ concatenated copies
of $\pgap(Q)$.  In $\pgap(Q)$, there are $\phi(Q)$ driving terms for the gap $g=2n$.  To see this, start
at any gap in $\pgap(Q)$ and proceed through the cycle $n_1$ times.  The length of each of these
driving terms is initially $n_1 \cdot \phi(Q)$.

We now want to bring this back into Eratosthenes sieve.  

Let $Q_0=Q$, and let $p_1, \ldots, p_k$ be the prime factors of $\pml{\bar{q}} / Q$.
For $i=1,\ldots, k$, let $Q_i = p_i \cdot Q_{i-1}$, with $Q_k = \pml{\bar{q}}$.  
In forming $\pgap(Q_i)$ from $\pgap(Q_{i-1})$,
we apply Corollary~\ref{qNCor}.  Since $p_i \not| g$, we have
\begin{equation*}
 \sum_{j=1}^{J} n_{2n,j}(Q_i) =  (p_i - 2) \cdot \sum_{j=1}^{J} n_{2n,j}(Q_{i-1})
 \end{equation*}
Thus at $p_k$ we have
\begin{eqnarray*}
 \sum_{j=1}^{J} n_{2n,j}(\pml{\bar{q}}) &=&  \sum_{j=1}^{J} n_{2n,j}(Q_k) \gap = \gap
 (p_k - 2) \cdot \sum_{j=1}^{J} n_{2n,j}(Q_{k-1}) \\
  & = & \left( \prod_{i=1}^k (p_i-2) \right) \sum_{j=1}^J n_{2n,j}(Q_0) 
= \left( \prod_{i=1}^k (p_i-2) \right) \phi(Q) 
\end{eqnarray*}
\end{proof}

\begin{proof} {\bf of Theorem~\ref{PolThm}.}
Let $g=2n$ be given.  Let $Q$ be the product of the prime factors dividing $g$ and
let $\bar{q}$ be the largest prime factor of $g$.  By Lemma~\ref{QLem} we know that
in $\pgap(\pml{\bar{q}})$ there occur driving terms for $g$ if not the gap $g$ itself.
Lemma~\ref{QLem} gives the total number of these driving terms as
$$ \sum_j n_{g,j}(\pml{\bar{q}})  
 = \phi(Q) \cdot \prod_{p < \bar{q}, \; p\; \nmid \; Q} (p-2).$$
The number of gaps $2$ in $\pgap(\pml{q})$ is
 $n_{2,1}(\pml{q}) = \prod_{2<p\le q} (p-2).$
 So for the ratios we have
\begin{eqnarray*}
\sum_j \Rat{j}{g}(\pml{\bar{q}}) & = &
 \sum_j n_{g,j}(\pml{\bar{q}}) / n_{2,1}(\pml{\bar{q}}) \\
 &=& \phi(Q) / \prod_{p | Q, \; p > 2} (p-2)  = \prod_{p | Q, \; p > 2} \frac{(p-1)}{(p-2)}.
 \end{eqnarray*}
 By Corollary~\ref{qNCor} and Corollary~\ref{InfCor}, we have the result
 $$ \Rat{1}{2n}(\infty) = \prod_{p | 2n, \; p > 2} \left( \frac{p-1}{p-2}\right).$$
\end{proof}

This establishes a strong analogue of Polignac's conjecture for Eratosthenes sieve.
Not only do all even numbers appear as gaps in later stages of the sieve, but they
do so in proportions that converge to specific ratios.  

We use the gap $g=2$ as the reference point since it
has no driving terms other than the gap itself. 
The gaps for other even numbers appear in ratios to $g=2$
implicit in the work of Hardy and Littlewood \cite{HL}.
In their Conjecture B, they predict that the number of gaps $g=2n$ that
occur for primes less than $N$ is approximately
$$ 2 C_2 \frac{N}{(\log N)^2} \prod_{p \neq 2, \; p | 2n} \frac{p-1}{p-2}.$$
We cannot yet predict how many of the gaps in a stage of Eratosthenes sieve
will survive subsequent stages of the sieve to be confirmed as gaps among primes.
However, we note that for $g=2$, the product in the above formula is $1$, and the ratio
of gaps $g=2n$ to gaps $2$ is given by this product.  

We have shown in Theorem~\ref{PolThm} that this same product describes the
asymptotic ratio of occurrences of the gap $g=2n$ to the gap $2$ in $\pgap(\pml{p})$
as $p \fto \infty$.  So if the survival of gaps in the sieve to be confirmed as gaps among
primes is at all fair, then we would expect this ratio of gaps in the sieve to be preserved
among gaps between primes.

\subsection{Examples from $\pgap(\pml{31})$}
To work with Theorem~\ref{PolThm} we look at some data from $\pgap(\pml{31})$.
In Table~\ref{G31Table} we exhibit part of the table for $\pgap(\pml{31})$, that gives the counts
$n_{g,j}$ of driving terms of length $j$ (columns) for various gaps $g$ (rows).  
The last two columns give the current sum of driving terms for each gap and the 
asymptotic value from Theorem~\ref{PolThm}.

\begin{table}[b]
\renewcommand\arraystretch{0.8}
\begin{center}
\begin{tabular}{|r|rrrrrrr|rr|} \hline
\multicolumn{1}{|c}{gap} 
& \multicolumn{7}{c}{ $n_{g,j}(31)$: driving terms of length $j$ in $\pgap(\pml{31})$ }
 & & \\ [1 ex]
\multicolumn{1}{|c}{ } & $\lil 3$ & $\lil 4$ & $\lil 5$ & $\lil 6$ & $\lil 7$ & $\lil 8$ & $\lil 9$
 & $\lil \sum \Rat{j}{g}$ & $\lil \Rat{1}{g}(\infty)$ \\ \hline 
$\lil g = 74$ & $\lil 1$ & $\lil 1206$ & $\lil 70194$ & $\lil 1550662$
 & $\lil 17523160$ & $\lil 113497678$ & $\lil 445136490$ & $\lil 1$ & $\lil 1.02857$ \\
$\lil 76$ &  & $\lil 602$ & $\lil 32194$ & $\lil 765488$
  & $\lil 9470176$ & $\lil 68041280$ & $\lil 302507798$ & $\lil 1.0588$ & $\lil 1.0588$\\
$\lil 78$ & & $\lil 292$ & $\lil 26060$ & $\lil 826426$
   & $\lil 12166908$ & $\lil 99284264$ & $\lil 489040926$ & $\lil 2.1818$  & $\lil 2.1818$ \\
$\lil 80$ &  & $\lil 2$ & $\lil 2876$ & $\lil 139926$ & $\lil 2656274$
    & $\lil 26634332$ & $\lil 159280176$  & $\lil 1.3333$  & $\lil 1.3333$\\
$\lil 82$ &  &  & $\lil 747$ & $\lil 46878$ & $\lil 1066848$
     & $\lil 12378176$ & $\lil 83484438$ & $\lil 1$ & $\lil 1.0256$ \\
$\lil 84$ & & $\lil 2$ & $\lil 1012$ & $\lil 58216$ & $\lil 1485176$
      & $\lil 18772184$ & $\lil 135450260$ & $\lil 2.4$ & $\lil 2.4$ \\
$\lil 86$ & & & $\lil 74$ & $\lil 4726$ & $\lil 147779$
       & $\lil 2453256$ & $\lil 23265268$ & $\lil 1$ & $\lil 1.0244$ \\
$\lil 88$ &  & & $\lil 2$ & $\lil 2190$ & $\lil 107182$
        & $\lil 2025910$ & $\lil 20603366$ & $\lil 1.1111$ & $\lil 1.1111$ \\
$\lil 90$ & & $\lil 8$ & $\lil 300$ & $\lil 9360$ & $\lil 195708$
         & $\lil 2829548$ & $\lil 26983182$ & $\lil 2.6667$ & $\lil 2.6667$ \\
$\lil 92$ & & & $\lil 20$ & $\lil 860$ & $\lil 26854$
          & $\lil 488854$ & $\lil 5364068$ & $\lil 1.0476$ & $\lil 1.0476$ \\
$\lil 94$  & & & $\lil 16$ & $\lil 740$ & $\lil 19740$
           & $\lil 333162$ & $\lil 3684805$ & $\lil 1$ & $\lil 1.0222$ \\
$\lil 96$ & &  & $\lil 4$ & $\lil 242$ & $\lil 9636$
            & $\lil 249610$ & $\lil 3693782$ & $\lil 2$ & $\lil 2$ \\
$\lil 98$ &  &  & & $\lil 28$ & $\lil 1482$ & $\lil 52328$
             & $\lil 968210$ & $\lil 1.2$ & $\lil 1.2$ \\
$\lil 100$ & & & & $\lil 8$ & $\lil 672$
              & $\lil 26428$ & $\lil 567560$ & $\lil 1.3333$ & $\lil 1.3333$ \\
$\lil 102$ &  & &  & & $\lil 78$ & $\lil 7042$
               & $\lil 249300$ & $\lil 2.133$ & $\lil 2.133$ \\
$\lil 104$ &  &  &  &  & $\lil 182$ & $\lil 6086$
                & $\lil 129016$ & $\lil 1.0909$ & $\lil 1.0909$ \\
$\lil 106$ &  & & &  & $\lil 16$ & $\lil 1168$
                 & $\lil 37144$ & $\lil 1$ & $\lil 1.0196$ \\
$\lil 108$ &  &  &  &  & $\lil 8$ & $\lil 1244$
                  & $\lil 44334$ & $\lil 2$ & $\lil 2$ \\
$\lil 110$ & & & & &  & $\lil 142$
                   & $\lil 7686$ & $\lil 1.4815$ & $\lil 1.4815$ \\
$\lil 112$  &  &  &  & &  & $\lil 68$
                    & $\lil 5294$ & $\lil 1.2$ & $\lil 1.2$ \\
$\lil 114$ & & & & &  & $\lil 22$
                     & $\lil 2388$ & $\lil 2.1176$ & $\lil 2.1176$ \\
$\lil 116$ &  & &  & & & $\lil 224$
                      & $\lil 4716$ & $\lil 1.0370$ & $\lil 1.0370$ \\
$\lil 118$ &  &  &  &  &  &  & $\lil 72$
                      & $\lil 1$ & $\lil 1.0175$ \\
$\lil 120$ & &  &  &  &  &  & $\lil 1012$ 
                      & $\lil 2.6667$ & $\lil 2.6667$ \\
$\lil 122$ & &  & &  &  &  & $\lil 70$
                      & $\lil 1$ & $\lil 1.0169$ \\
$\lil 124$ & &  & &  &  &  & $\lil 28$
                      & $\lil 1.0345$ & $\lil 1.0345$ \\
$\lil 126$ & &  & &  &  &  & $\lil 4$
                      & $\lil 2.4$ & $\lil 2.4$ \\
$\lil 128$ & &  & &  &  &  & 
                      & $\lil 1$ & $\lil 1$ \\
$\lil 130$ & &  & &  &  &  & 
                      & $\lil 1.4545$ & $\lil 1.4545$ \\
$\lil 132$ & &  & &  &  &  & $\lil 2$
                      & $\lil 2.2222$ & $\lil 2.2222$ \\
\hline
 \end{tabular}
 \caption{ \label{G31Table} A sample of the population data for gaps $g$ and their driving terms
 in the cycle of gaps $\pgap(\pml{31})$.  This section of the table records the data where the driving
 terms of length $9$ are running out.  For the range of gaps displayed, there are no nonzero entries
 for $j=1, \; 2$.  The last two columns list for each gap the ratio of the sum of all the driving terms
 in $\pgap(\pml{31})$ to the population $g=2$, and the asymptotic ratio.}
 \end{center}
 \end{table}
\renewcommand\arraystretch{1}

In each stage of Eratosthenes sieve, some copies of the driving terms of length $j$ will have
at least one interior closure, resulting in shorter driving terms at the next stage.  For this part of the table, 
$g \ge 2p_{k+1}$ and so more than one closure could occur within a single copy of a driving
term.  

Regarding our work on Polignac's conjecture, from Table~\ref{G31Table} we observe that with
$p_0 =31$, if a gap $g=2n$ has a driving term of length $j$, then at each ensuing stage of the sieve a
shorter driving term will be produced.  Thus the gap itself will occur in
$\pgap(\pml{p_k})$ for $k \le \min j-1$, the length of the shortest driving term for $g$
in $\pgap(\pml{31})$.

We have chosen the
part of the table at which the
driving terms through length $9$ are running out.  In this part of the table we observe
interesting patterns for the maximum gap associated with driving terms of a given length.
The driving terms of length $4$ have sums up to $90$ but none of sums $82$, $86$, or $88$.
Interestingly, although the gap $128$ is a power of $2$, in $\pgap(\pml{31})$ its 
driving terms span the lengths from $11$ to $27$; yet the gaps $g=126$ and $g=132$
already have driving terms of length $9$.

From the tabled values for $\pgap(\pml{31})$, we see that the driving term of length $3$
for $g=74$ will advance into an actual gap in two more stages of the sieve.  Thus the maximum
gap in $\pgap(\pml{41})$ is at least $74$, and the maximum gap for
$\pgap(\pml{43})$ is at least $90$.

Note that in Table~\ref{G31Table}, some gaps
have not attained their asymptotic ratios:
$$\sum_j \Rat{j}{g}(\pml{31}) \neq \Rat{1}{g}(\infty) \gap {\rm for} \gap
g=74, 82, 86, 94,106, 118, 122.$$
Up through $\pgap(\pml{31})$
these ratios are $1$, but for each gap, we know that this ratio will jump
to equal $\Rat{1}{g}(\infty)$ in the respective $\pgap(\pml{\bar{q}})$.  
How does the ratio transition from $1$ to the
asymptotic value?  If we look further in the data for $\pgap(\pml{31})$, we find that for
the gap $g=222$, $\sum_j \Rat{j}{222}(\pml{31})=2$ but the asymptotic value is
$\Rat{1}{222}(\infty) = 72/35.$

These gaps $g=2n$ have maximum prime divisor $\bar{q}$ greater than the prime $p$ for
the current stage of the sieve $\pgap(\pml{p})$.  From Corollary~\ref{qNCor} and the
approach to proving Lemma~\ref{QLem}, we are able to establish the following.

\begin{corollary}
Let $g=2n$, and let $Q=q_1 q_2 \cdots q_k$ be the product of the distinct prime factors of $g$,
with $q_1 < q_2 < \cdots < q_k$.  Then for $\pgap(\pml{p})$,
$$
\sum_j \Rat{j}{g}(\pml{p}) = \prod_{2 < q_i \le p} \left(\frac{q_i-1}{q_i-2}\right).
$$
\end{corollary}

\begin{proof}
Let $p=q_j$ for one of the prime factors in $Q$.  By Corollary~\ref{qNCor} these are the only
values of $p$ at which the sum of the ratios $\sum_j \Rat{j}{g}(p)$ changes.

Let $Q_j = q_1 q_2 \cdot q_j$.
In $\pgap(\pml{q_j})$, $g$ behaves like a multiple of $Q_j$. 
As in the proof of Lemma~\ref{QLem}, in $\pgap(Q_j)$ each generator begins a driving term
of sum $2n$, consisting of $2n / Q_j$ complete cycles.  There are $\phi(Q_j)$ such driving
terms.  

We complete $\pgap(\pml{q_j})$ as before by introducing the missing prime factors.  
The other prime factors do not
divide $2n$, and so by Corollary~\ref{qNCor} the sum of the ratios is unchanged by these factors.
We have our result:
$$ \sum_j \Rat{j}{g}(\pml{q_j}) = \prod_{2 < q_i \le q_j} \left(\frac{q_i-1}{q_i-2}\right).
$$
\end{proof}

For the gap itself, we know from Equation~\ref{EqEigSys} that the ratio $\Rat{1}{g}(\pml{p})$ 
converges to its asymptotic value
as quickly as $a_2^k \fto 0$.  We have observed above that this convergence is slow.

\section{Gaps between prime numbers and gaps in the sieve}
In our work above, we obtain several exact and asymptotic results regarding the cycles of gaps
$\pgap(\pml{p})$ that occur in Eratosthenes sieve.  What is the relationship between
the cycle of gaps  and the gaps between prime numbers?

Let's look at $\pgap(\pml{7})$ as an example. This cycle of gaps has length $48$, and the
gaps sum to $210$.
$$
\pgap(\pml{7}) =
 {\scriptstyle 
 10, 2424626424 6 62642 6 46 8 42424 8 64 6 24626 6 4246264242, 10, 2}
 $$
The first gap $10$ marks the next prime, $p_{k+1}=11$.  This first gap is the accumulation of gaps
between the primes from $1$ to $p_{k+1}$.  The next several gaps will actually survive to be confirmed
as gaps between primes, since the smallest remaining closure will occur at $p_{k+1}^2 = 121$.  

In $\pgap(\pml{p_k})$ all of the gaps from $p_{k+1}$ until the gap before $p_{k+1}^2$ 
are actually gaps between primes. Then, after closing
at $p_{k+1}^2$, the next set of gaps survive up until the closure at $p_{k+1}\cdot p_{k+2}$.
Let us look at the closures
that occur in $\pgap(\pml{7})$ as the sieve continues, marking the gaps that survive in bold.
$$\begin{array}{rl}
\pgap(\pml{7}) \;  = &
 {\scriptstyle 
 10, 2424626424 6 62642 6 46 8 42424 8 64 6 24626 6 4246264242, 10, 2} \\
(p=11)  \Rightarrow & {\scriptstyle 
 10, +
 \overbrace{\scriptstyle {\bf 2424626424 6 62642 6 46 8 42424} \; 8}^{110} + 
 \overbrace{\scriptstyle 6 \; {\bf 4 6 2} \; 4}^{22} +
 \overbrace{\scriptstyle 6 \; 26 6 42462 \; 6}^{44} +
 \overbrace{\scriptstyle 4 \; 242, \; 10,}^{22} + 2  \ldots} \\
(p= 13)  \Rightarrow &  {\scriptstyle 12, +
 \overbrace{\scriptstyle {\bf 424626424 6 62642 6 46 8 42424, \; 14, \;  462, \; 10, \; 26 6 4} \; 2}^{156} + 
 4 {\bf 62, \; 10,  \; 242, \; 12,} \ldots}  \\
 (p = 17)  \Rightarrow &  {\scriptstyle 
 16, \; {\bf 24626424 6 62642 6 46 8 42424, \; 14, \;  462, \; 10, \; 26 6 4 \; 6\; 62, \; 10,  \; 242, \; 12,} \ldots }  
 \end{array}
 $$
From the prime $p=17$ and up, there are no more closures for this sequence of gaps.  All of the
remaining gaps survive as gaps between primes.

All of the gaps between primes are generated out of these cycles of gaps, with the gaps at the front of 
the cycle surviving subsequent closures.  

We have some evidence that the recursion is a fair process.  There is an approximate uniformity
to the replication.  Each instance of a gap in $\pgap(\pml{p_k})$ is replicated $p_{k+1}$ times uniformly
spaced in step R2, and then two of these copies are removed through closures.  Also,
the parameters for the dynamic system are independent of the size of the gap; each constellation 
of length $j$ is treated the same, with the threshold condition $g < 2p_{k+1}$.  If the recursion
is a fair process, then do we expect the survival of gaps to be fair as well?

If we had a better characterization of the survival of the gaps in $\pgap(\pml{p})$, or of the 
distribution of subsequent closures across this cycle of gaps, we would be able to make 
stronger statements about what these exact results on the gaps in Eratosthenes sieve imply
about the gaps between primes.  

\section{Conclusion}
By identifying structure among the gaps in each stage of Eratosthenes sieve, we 
have been able to develop an exact model for the populations of gaps and their
driving terms across stages of the sieve.  
We have developed a model for a discrete dynamic system that takes 
the initial populations of a gap $g$ and all its driving terms in a cycle of gaps 
$\pgap(\pml{p_0})$ such that $g < 2p_1$, and thereafter the model provides the exact
populations of this gap and its driving terms through all subsequent cycles of gaps.

The coefficients of this model
do not depend on the specific gap, only on the prime for each stage of the sieve.
To this extent, the the sieve is agnostic to the size of the gaps.

On the other hand, the initial conditions for the model do depend on the size
of the gap.  More precisely, the initial conditions depend on the prime factorization of
the gap.

For several conjectures about the gaps between primes, we can offer 
precise results for their analogues in the cycles of gaps across stages of Eratosthenes
sieve.  Foremost among these analogues, perhaps, is that we are able to 
affirm in Theorem~\ref{PolThm} an analogue of Polignac's conjecture 
that also supports Hardy \& Littlewood's
Conjecture B:

{\em For any even number $2n$, the gap $g=2n$ arises in Eratosthenes sieve, and
as $p \fto \infty$, the number of occurrences of the gap $g=2n$ to the gap $2$
approaches the ratio
$$w_{2n,1}(\infty) =  \prod_{q>2, \; q|n} \frac{q-1}{q-2}.
$$}
These results
provide evidence toward the original conjectures, to the extent that gaps
in stages of Eratosthenes sieve are indicative of gaps among primes themselves.

To obtain the analogue of Polignac's conjecture, we had to generalize our approach,
looking at the cycles of gaps $\pgap(N)$ for any $N$ and leveraging the simplicity of the
dominant right and left eigenvectors for the dynamic system, 
corresponding to eigenvalue $1$.

It is daunting to consider the span of these cycles of gaps $\pgap(\pml{p})$.
This cycle of gaps $\pgap(\pml{p})$ has $\phi(\pml{p})$ gaps that sum up to
$\pml{p}$.  For example, we have calculated initial conditions for gaps in 
$\pgap(\pml{31})$, which consists of about $3 \times 10^{10}$ gaps whose sum
is around $2 \times 10^{11}$.  

The cycle $\pgap(\pml{31})$ completely determines the sequence of gaps between
the primes from $37$ up to $37^2=1369$, and it sets the number and location of all
the driving terms up through $2 \times 10^{11}$.
This is all determined by the time we have run Eratosthenes sieve only through $p=31$.

For this paper, our analysis of the dynamic system has focused on the populations of 
the gaps.  We note that the dynamic system can be applied to constellations as well,
providing analogues to complement works on constellations of primes 
\cite{HL, quads, GPY, GranPatterns}.
Once a constellation $s$ of length $j$ and sum $g$ arises, if $j < p-1$, then this constellation
persists through all later cycles of gaps and its population grows.  This raises the prospect,
for example, of  finding twin primes infinitely often in the constellations $242$, and $2,10,2$, and even
$2,10,2,10,2$.  Corollary~\ref{CorGrowth} describes the growth of all sufficiently small constellations
within the sieve.


\bibliographystyle{amsplain}

\providecommand{\bysame}{\leavevmode\hbox to3em{\hrulefill}\thinspace}
\providecommand{\MR}{\relax\ifhmode\unskip\space\fi MR }
\providecommand{\MRhref}[2]{%
  \href{http://www.ams.org/mathscinet-getitem?mr=#1}{#2}
}
\providecommand{\href}[2]{#2}

\end{document}